\documentclass[12pt]{amsart}
\usepackage{t1enc}
\usepackage[latin2]{inputenc}
\usepackage{verbatim}
\usepackage{amsmath,amsfonts,amssymb,amsthm}
\usepackage[mathcal]{eucal}
\usepackage{enumerate}
\usepackage{graphics}
 \newtheorem{thm}{Theorem}[section]
 \newtheorem{cor}[thm]{Corollary}
 \newtheorem{lem}[thm]{Lemma}
 
 \theoremstyle{definition}
 
 \theoremstyle{remark}
 
 \numberwithin{equation}{subsection}

 \DeclareMathOperator{\mes}{mes}
 \DeclareMathOperator{\supp}{supp}
 \DeclareMathOperator{\Mod}{mod}

 \newcommand{\I}{\mathcal{I}}
 \newcommand{\N}{\mathbb{N}}
 \newcommand{\Real}{\mathbb{R}}
  
  \newcommand{\m}[1]{\mes\left\{#1\right\}}
  
 \newcommand{\set}[1]{\left\{#1\right\}}

 \newcommand{\F}[1]{\mathcal{F}_{#1}}

\begin{document}

\title[Ces\`{a}ro means of subsequences of partial sums]{Ces\`{a}ro means of subsequences of partial sums of trigonometric Fourier series}

\author{Gy\"orgy G\'at}

\address{Institute of Mathematics, University of Debrecen, H-4002 Debrecen, Pf. 400, Hungary}
\email{gat.gyorgy@science.unideb.hu}

\thanks{Research supported by the
Hungarian National Foundation for Scientific Research (OTKA),
grant no.  K111651 and by project EFOP-3.6.1-16-2016-00022 supported by the European Union, co-financed by the European Social Fund.}

\date{}

\begin{abstract}
In 1936 Zygmunt Zalcwasser  asked  with respect to the trigonometric system that how ``rare'' can a sequence of strictly monotone increasing integers $(n_j)$ be such that the almost everywhere relation
$\frac{1}{N}\sum_{j=1}^N S_{n_j}f \to f$ is fulfilled for each integrable function $f$. In this paper, we give an answer to this question.
It follows from the main result that this a.e. relation holds for every integrable function $f$ and  lacunary sequence $(n_j)$ of natural numbers.
\end{abstract}

\maketitle

\section{Introduction and the main theorem}

In 1936 Zalcwasser \cite{zal} asked how ``rare'' can a sequence of integers $(n_j)$ be such that
\[
\frac{1}{N}\sum_{j=1}^N S_{n_j}f \to f
  \]
a.e. for every function $f\in L^1$. In this paper, we give an answer to this question (Theorem \ref{main}).

It is of main interest in the theory of trigonometric Fourier series that how to reconstruct the function from the partial sums of its Fourier series.
It is known from Du Bois-Reymond \cite{dub} that the Fourier series of a continuous function can unboundedly diverge at some point.

A. N. Kolmogoroff \cite{kol1} constructed his famous example of a function $f\in L^1$
such that the partial sums  $S_mf(x)$ diverge unboundedly almost everywhere. In another paper \cite{kol2}
he constructed an everywhere divergent Fourier series.
In particular, it was not clear whether the Fourier series of a continuous function
can diverge everywhere. Carleson \cite{car} showed that if $f \in  L^2$, then the partial sums converge to the function almost everywhere.
The condition $f \in  L^2$ in the Carleson theorem was weakened by Hunt \cite{hunt} ($f\in L^p\, (p>1)$) and recently
Antonov \cite{ant} who proved that if $f$ is in the class $L\log^+L \log^+\log^+\log^+ L$, then  the partial sums of the Fourier series converge to the function almost everywhere again. It is a fundamental question how to reconstruct a function in  $L^1$ from its Fourier series.

Inspired by Fej\'er's results, Lebesgue  \cite{leb}  showed that for each integrable function  we have the almost everywhere convergence of the Fej\'er means $\sigma_nf = \frac{1}{n+1}\sum_{m=0}^n S_mf\to f$.

It is also of prior interest  what it can be said - with respect to this reconstruction issue - if we have only a subsequence of the partial sums.
With respect to the partial sums and the Lebesgue space $L^1$ the bad news is that in 1982 Totik  \cite{totik} showed that for each subsequence $(n_j)$ of the sequence of natural numbers there exists an integrable function $f$ such that $\sup_{j}|S_{n_j}f|=+\infty$ everywhere. Moreover,
Konyagin  \cite{konT} proved that for any increasing sequence $(n_j)$ of positive integers and any nondecreasing function
$\phi: [0, +\infty) \to [0, +\infty)$ satisfying the condition $\phi(u)=o(u\log\log u)$, there is a function
$f \in\phi(L)$ such that $\sup_{j}|S_{n_j}f|=+\infty$ everywhere. That is, a summation method is needed.

In 1936 Zalcwasser \cite{zal} asked how ``rare''  the sequence of integers $(n_j)$ can  be such that
\[
\frac{1}{N}\sum_{j=1}^N S_{n_j}f \to f
  \]
for every function $f\in L^1$.
This problem  was completely solved with respect to the trigonometric system for continuous functions and uniform convergence
in \cite{sal, tz, bel_anal, car2}. That is, if the sequence $(n_j)$ is convex, then the condition $\sup_j j^{-1/2}\log n_j <+\infty$ is necessary and
sufficient for the uniform convergence for every continuous function.

With respect to convergence almost everywhere and integrable functions the situation is more complicated. In 1936 Zalcwasser \cite{zal} proved the a.e. relation
$\frac{1}{N}\sum_{j=1}^N S_{j^2}f \to f$ for each integrable function $f$. In his paper Salem \cite[page 394]{sal} writes that this theorem of Zalcwasser is extended to $j^3$ and $j^4$ but
there is no citation in \cite{sal} about it.
Belinsky proved \cite{bel_pams} the existence of a sequence $n_j\sim \exp(\sqrt[3]{j})$ such that the relation $\frac{1}{N}\sum_{j=1}^{N}S_{n_j}f\to f$ holds a.e. for every integrable function. In this paper, Belinsky also conjectured that if the sequence $(n_j)$ is convex, then the condition $\sup_j j^{-1/2}\log n_j <+\infty$ is necessary and sufficient again. So, that would be the answer to the problem of Zalcwasser \cite{zal} in this case (trigonometric system, a.e. convergence and $L^1$ functions).
In this paper, - among others - it is proved that this is not the case.

The system of functions
\(
e^{\imath nx} \quad (n=0, \pm 1, \pm 2, \dots)
\)
($x\in \Real, \imath = \sqrt{-1}$) is called the trigonometric
system. It is orthogonal over any interval of length $2\pi$,
specifically over $T:= [-\pi, \pi)$. 
Let $f\in L^1(T)$, that is $f$ is an integrable function on $T$. The $k$th Fourier
coefficient of $f$ is
\[
\hat f(k) := \frac{1}{2\pi}\int_{T}f(x) e^{-\imath kt} dt,
\]
where $k$ is any integer number. The $n$th ($n\in\N$) partial sum
of the Fourier series of $f$ is
\[
S_nf(y) := \sum_{k=-n}^{n}\hat f(k)e^{\imath ky}.
\]
The $n$th ($n\in\N$) Fej\'er or $(C,1)$ mean of the function $f$ is
defined in the following way:
\[
\sigma_nf(y) := \frac{1}{n+1}\sum_{k=0}^n S_kf(y).
\]
It is known that
\[
\sigma_nf(y) = \frac{1}{\pi}\int_Tf(x)K_n(y-x) dx,
\]
where the function $K_n$ is known as the $n$th Fej\'er kernel; we
will now find an appropriate expression for it (see e.g. the book
of Bary \cite{B}), namely 
\[
K_n(u) = \frac{1}{2(n+1)}\left(
\frac{\sin(\frac{u}{2}(n+1))}{\sin(\frac{u}{2})} \right)^2.
\]
From this expression one can immediately derive the following
properties of the kernel. They will play an essential role later.
\[
K_n(u) \ge 0.
\]
\[
K_n(u) \le \frac{\pi^2}{2(n+1)u^2} \quad (0< |u| \le \pi).
\]

Now, we state the main theorem of the paper.
\begin{thm}\label{main} Let $f\in L^1(T)$ be a function and $(n_j)$ be a sequence of natural numbers with the property that
$n_{j+1} \ge \left(1+\frac{1}{j^{\delta}}\right)n_j$ holds for $j\in\mathbb{N}$ and for some $0<\delta < 1/2$. Then the almost everywhere relation
\[
\lim_{N\to \infty}\frac{1}{N}\sum_{j=1}^{N} S_{n_j}f = f
\]
holds on $T$.
\end{thm}
\begin{cor}\label{main_cor} Let $(n_j)$ be a lacunary sequence of natural numbers. Then the almost everywhere relation
\(\lim_{N\to \infty}\frac{1}{N}\sum_{j=1}^{N} S_{n_j}f = f \)
holds for every $f\in L^1(T)$.
\end{cor}

We remark, that the corresponding version of Corollary \ref{main_cor} for the Walsh-Paley system can be found in \cite{gatJAT}. On the other hand, no part of the paper \cite{gatJAT} could be used here, but the fact that the author of \cite{gatJAT} could prove Corollary \ref{main_cor} for the Walsh system was an inspiration for  this article.

We say some preliminary words about the main ideas of the proof of Theorem \ref{main}.
Let $V_nf :=\frac{1}{n}\sum_{j=n}^{2n-1}S_jf$ be the $n$th  de la Vall\'ee-Poussin mean of the integrable function $f$.
Instead of the $(C,1)$ means of $S_{n_j}f$ we can investigate $(C,1)$ means of $S_{n_j}f-V_{n_j}f$, since for the de la Vall\'ee-Poussin means $V_{n_j}f$
we have the a.e. convergence $V_{n_j}f\to f$. We define some sequence of sets $\beta F_{\beta_j}$ and the whole
sections $3,4$ and $5$ are dedicated to prove the ``orthogonality'' lemma (Lemma \ref{orthogonality}). That is,
\[
\left\|\sum_{j=1}^{N}\left( S_{n_j}f-V_{n_j}f\right)  \left(\sigma_{m_j}1_{\overline{\beta F_{\beta_j}}}\right)\right\|_2^2
 \le C_{\beta} N\log^5 (N+1)\|f\|_1\lambda,
\]
where constant $C_{\beta}$ depends only on $\beta$ and it is uniform in $f,\lambda, N$ and $(n_j)$. To be honest, ``orthogonality'' lemma (Lemma \ref{orthogonality}) is proved for lacunary sequences $(n_j)$ (with quotient greater than $2.5$), but this fact will be
got around in the proof of the main theorem (Theorem \ref{main}).
In the ``replacement'' lemma (Lemma \ref{replacement}) we basically prove that the sequences
\[
\frac{1}{N}\sum_{j=1}^{N}\left( S_{n_j}f-V_{n_j}f\right)  \left(\sigma_{m_j}1_{\overline{\beta F_{\beta_j}}}\right),
\quad
\frac{1}{N}\sum_{j=1}^{N}
\left( S_{n_j}f-V_{n_j}f\right)
\]
are a.e. equiconvergent. This will lead to the a.e. convergence
$\frac{1}{N}\sum_{j=1}^{N}\left( S_{n_j}f-V_{n_j}f\right)\to 0$.
Roughly speaking, the proof of the main theorem is based on the ``orthogonality'' and  ``replacement'' lemmas.
In the sequel, we start the process on the way to the ``orthogonality'' lemma.

\section{A Decomposition Lemma}

The dyadic subintervals of $T$ are defined in the following way.
\[
\I_0 := \set{T}, \quad \I_1 := \set{[-\pi,0), [0,\pi)},
\]
\[
\I_2 :=  \set{[-\pi,-\pi/2), [-\pi/2,0), [0,\pi/2), [\pi/2,\pi)},
\dots
\]

\[
\I := \bigcup_{n=0}^{\infty}\I_n.
\]

The elements of $\I$ are said to be  dyadic intervals. If $F\in
\I$, then there exists a unique $n\in \mathbb{N}$ such that $F\in
\I_n$, and consequently $\mes(F)= |F| = \frac{2\pi}{2^n}$. Each $\I_n$ has
$2^n$ disjoint elements ($n\in\N$).

The following Calderon-Zygmund type decomposition
lemma can be found for instance in \cite[page 17]{stein} or  \cite[page 90]{SWS} (more precisely, in a slightly different way)
or in \cite{gatAMAPN} (with an elementary proof). This will play a prominent role in the proof of
the main theorem of this paper.

\begin{lem}\label{cz}
Let $f\in L^1(T)$, and $\lambda>\|f\|_1/(2\pi)$ . Then there
exists a sequence of integrable functions $(f_i)$ such that
\[
f = \sum_{i=0}^{\infty}f_i \quad \mbox{a.e.},
\]
\[
\|f_0\|_{\infty} \le 2\lambda, \quad \|f_0\|_{1} \le 2\|f\|_{1},
\quad \mbox{and}
\]
\[
\supp f_i \subset I^{i}, \quad \mbox{where}
\]
$I^{i}\in \I$ are disjoint dyadic intervals depending only on $|f|$ (and $\lambda$),
\[
\mes (I^{i}) = \frac{2\pi}{2^{k_i}} \quad \mbox{for some}
\]
$k_i \ge 1\,  (i\ge 1)$. Moreover,
$\int_{T} f_i(x) dx=\int_{I^i} f_i(x) dx=0\, (i\ge 1)$,
\[
\lambda < \frac{1}{\mes(I^{i})}\int_{I^{i}}|f|\le 2\lambda, \quad\frac{1}{\mes(I^{i})}\int_{I^{i}}|f_i| \le 4\lambda
\]
and for the union
\[
F := \bigcup_{i=1}^{\infty}I^{i}
\]
of the disjoint dyadic intervals $I^{i}$  ($i\ge 1$)
we have $\mes(F) = |F|  \le \|f\|_{1}/\lambda$.
\end{lem}
Also using the notation with respect to Lemma \ref{cz} we define $\F{} := \set{I^{i} : i=1,\dots,}$. That is, $\F{}$ is the set of dyadic intervals, whose union  is the set $F$.
Moreover, for any dyadic interval $I$, $I\in \F{}$ if and only if $|I|^{-1}\int_{I}|f| >\lambda$ and $|J|^{-1}\int_{J}|f| \le \lambda$ for every dyadic interval $J\supsetneq I$.
For any positive number $\beta>0$ let $\F{\beta} := \set{I\in \F{} : |I|> \beta}$, $F_{\beta} := \cup_{I\in \F{\beta}} I$.
Besides, remark that $f_i = (f - |I^i|^{-1}\int_{I^i}f)1_{I^i}$ ($i\ge 1$), $f_0 = f1_{\bar F} + \sum_{i=1}^{\infty}(|I^i|^{-1}\int_{I^i}f)  1_{I^i}$ ($1_{\bar F}$ is the characteristic function of the complement of  $F$).

Let $I_n(x)$ be the dyadic interval of measure $2\pi/2^n$ with $x\in I_n(x)$.
For any dyadic interval $I\in \mathcal{I}$ and integer $i$ let $I^{(i)} := I + |I|i$ be the $i$th neighbour
of $I$ (the addition $I+|I|i$ is done  modulo $T=[-\pi, \pi)$, that is a circle represents $T$). For instance $I^{0} = I$ and
for $I=I_n(x)\in\mathcal{I}$, $I^{(1)} = I_n^{(1)}(x) := I_n(x)+ 2\pi/2^n, I^{(-1)} = I_n^{(-1)}(x) := I_n(x)- 2\pi/2^n$
are the right and left adjacent (modulo $T$) dyadic intervals of $I_n(x)$ with the same measure. We also use the notation $I^+ := I^{(1)}, I^- := I^{(-1)}$.

For any $I_n(x)\in\mathcal{I}$ let $3I_n(x) := \bigcup_{i=-1}^1 I^{(i)}_n(x)$ be the tripled of $I_n(x)$. That is, $3I_n(x)$ is an interval with the same center as $I_n(x)$ and  with the tripled measure of $I_n(x)$. Similarly, we define
\[
5I_n(x) := \bigcup_{i=-2}^2 I^{(i)}_n(x)
\]
and so on. That is, we can define $7I_n(x), 9I_n(x)\dots$.
Moreover, use the notation for any odd integer $\gamma$
\[
\gamma F_{\beta} := \bigcup_{I\in \F{\beta}}\gamma I, \quad \gamma F :=  \bigcup_{I\in \F{}} \gamma I.
\]

We need a preliminary lemma:
\begin{lem}\label{gammaIgammaJ} Let $\gamma$ be an odd natural number, $J$ be a dyadic interval. Then we  have
\[
\sum_{I^{(i)}\subset J, I\in \F{}} |I^{(i)}| \le C_{\gamma}|J| \quad \mbox{for}\quad |i| \le \frac{\gamma-1}{2}
\]
and
\[
\sum_{I,J\in \F{}} |\gamma I\cap \gamma J| \le C_{\gamma}|F|.
\]
\end{lem}
\begin{proof}
   Without loss of generality we can suppose that $i\ge 0$. If $I^{(i)}\subset J$, then $|I^{(i)}|\le  |J|$ and $I = I^{(0)} = (I^{(i)})^{(-i)} \subset\bigcup_{j=0}^{i} J^{(-j)}\subset \gamma J$. The sets $I\in \F{}$ are disjoint and consequently
  \[
  \sum_{I^{(i)}\subset J, I\in \F{}} |I^{(i)}|  = \sum_{I^{(i)}\subset J, I\in \F{}} |I| \le
  \sum_{I\subset \gamma J, I\in \F{}} |I| \le |\gamma J| = \gamma |J|.
  \]
  Now, turn our attention to $\sum_{I,J\in \F{}} |\gamma I\cap \gamma J|$.
  Suppose that we take the sum for pairs $I,J$ satisfying $|I|\le |J|$ and take  fixed $i,j\in \{1/2-\gamma/2, \dots, \gamma/2-1/2\}$.

  The intersection $I^{(i)}\cap J^{(j)}$ of dyadic intervals  can  be different from $\emptyset$ if and only if
  $I^{(i)}\subset J^{(j)}$. By the first inequality of Lemma \ref{gammaIgammaJ} proved already we have
  \[
  \sum_{J\in \F{}}\sum_{I\in \F{}, |I|\le |J|}|I^{(i)}\cap J^{(j)}| \le C_{\gamma} \sum_{J\in \F{}}|J| \le C_{\gamma} |F|.
  \]
Taking into account that $\gamma I\cap \gamma J = \bigcup_{i,j=1/2-\gamma/2}^{\gamma/2-1/2} (I^{(i)}\cap J^{(j)})$ and the fact that
the case $|J|\le |I|$ can be discussed in the same way  the proof of Lemma \ref{gammaIgammaJ} is  complete.
\end{proof}

\section{The integral on $\gamma F_{\epsilon}$}

For $f\in L^1$ and $y\in T$ set a version of  Hilbert transforms of $f$ at $y$ for $n\in \N, |n| := \lfloor\log_2 n \rfloor$ (that is, $2^{|n|}\le n < 2^{|n|+1}$) as:
\begin{equation}\label{hilbert_def}
H_n f(y) := \int_{T\setminus 3I_{|n|}(y)} f(x)\cot\left(\frac{y-x}{2}\right) dx.
\end{equation}

The transform $H_n$ is of type $(L^p, L^p)$ for any $1<p<\infty$ and it is also of weak type $(L^1, L^1)$ as it can be seen in the following way.
Let $\delta = |I_{|n|}(y)| = 2\pi/2^{|n|}$. Then  $(y-\delta, y+\delta)\subset 3I_{|n|}(y)$. Then we have
\[
\begin{split}
& |H_n f(y)| \le \left|\int_{T\setminus (y-\delta, y+\delta)} f(x)\cot\left(\frac{y-x}{2}\right) dx\right| + Cn\int_{3I_{|n|\setminus}(y)\setminus (y-\delta, y+\delta)} |f(x)| dx
\\
& \le \sup_{\epsilon>0}\left|\int_{T\setminus (y-\epsilon, y+\epsilon)} f(x)\cot\left(\frac{y-x}{2}\right) dx\right| + \frac{C}{\delta}\int_{y-2\delta}^{y+2\delta}|f(x)| dx.
\end{split}
\]
Since the maximal (``ordinary'') Hilbert transform (\cite[Chapter 3]{BS}) and the integral mean value operator are of type $(L^p, L^p)$ for any $1<p<\infty$ and they are also of weak type $(L^1, L^1)$, then so is the ``newly defined'' Hilbert transform.

The first lemma needed is:
\begin{lem}\label{gammaFjH2}
Let $\beta > \gamma >5$ be odd integers, $\epsilon>0$, $n,m\in\N$ with $n\le 100m$ and $f\in L^1(T), \lambda >\|f\|_1/(2\pi)$. Then the inequality
\[
\int_{\gamma F_{\epsilon}}|H_n f(y)|^2 |\sigma_m 1_{\overline{\beta F_{\epsilon}}}(y)|^2 dy \le C_{\beta, \gamma}\|f\|_1\lambda
\]
holds, where the constant $C_{\beta, \gamma}$ can depend only on $\beta$ and $\gamma$ (it is uniform in $f, n, m, \epsilon, \lambda$) and $1_{\overline{\beta F_{\epsilon}}}$ is the characteristic function of the complement
of $\beta F_{\epsilon}$.
\end{lem}
\begin{proof}
Without loss of generality, in order not to write too many conjugate signs, we suppose that $f$ is a real function.
  By applying the Calderon-Zygmund decomposition lemma (Lemma \ref{cz}) for $f$ we have
  $f = \sum_{i=0}^{\infty}f_i$, where $\|f_0\|_{\infty}\le 2\lambda, \|f_0\|_1 \le 2\|f\|_1$. Since the Hilbert transform is of type $(L^2, L^2)$ and
  $\|\sigma_m 1_{\overline{\beta F_{\epsilon}}}\|_{\infty}\le 1$, then we have
\[
\int_{\gamma F_{\epsilon}}|H_n f_0(y)|^2 |\sigma_m 1_{\overline{\beta F_{\epsilon}}}(y)|^2 dy \le \|H_n f_0(y)\|^2_2 \le C\|f_0\|_2^2 \le C\lambda\|f\|_1.
\]
That is, instead of $f$ we have to investigate only $f^0 :=  \sum_{i=1}^{\infty}f_i$.
Since the elements of $\F{}$ are disjoint dyadic intervals, therefore
\[
H_nf^0(y) = \sum_{J\in \F{}}\int_{J\setminus 3I_{|n|}(y)} f^0(x)\cot\left(\frac{y-x}{2}\right) dx,
\]
and then

\[
\begin{split}
     & \int_{\gamma F_{\epsilon}}|H_n f^0(y)|^2 |\sigma_m 1_{\overline{\beta F_{\epsilon}}}(y)|^2 dy\\
& = \sum_{(J,K)\in  \F{}\times \F{}}\int_{\gamma F_{\epsilon}}
\int_{J\setminus 3I_{|n|}(y)} f^0(x)\cot\left(\frac{y-x}{2}\right) dx\\
& \times\int_{K\setminus 3I_{|n|}(y)} f^0(z)\cot\left(\frac{y-z}{2}\right) dz
  |\sigma_m 1_{\overline{\beta F_{\epsilon}}}(y)|^2 dy.
\end{split}
\]
The sum over pairs $(J,K)\in  \F{}\times \F{}$ will be divided into the following parts: ($\mathcal{G}=  \F{}\times \F{}$)

\[
\begin{split}
      &  \sum_{(J,K)\in \mathcal{G}} \int_{\gamma F_{\epsilon}} = \sum_{(J,K)\in \mathcal{G}} \int_{\gamma F_{\epsilon}\cap\gamma J \cap \gamma K} \\
& +  \sum_{(J,K)\in \mathcal{G}} \int_{(\gamma F_{\epsilon}\cap\gamma J) \setminus \gamma K}
 +  \sum_{(J,K)\in \mathcal{G}} \int_{(\gamma F_{\epsilon}\cap\gamma K) \setminus \gamma J}\\
 & +  \sum_{(J,K)\in \mathcal{G}} \int_{\gamma F_{\epsilon}\setminus(\gamma J\cup \gamma K)}
 =: \sum_{i=1}^{4}A^i.
\end{split}
\]

First, we investigate $A^1$. For $x\notin 3I_{|n|}(y)$  we have
\[
\left|\int_{J\setminus 3I_{|n|}(y)} f^0(x)\cot\left(\frac{y-x}{2}\right) dx\right| \le Cn \int_J |f^0(x)| dx \le Cn |J| |J|^{-1}\int_J |f^0| \le C\lambda n|J|
\]
as it comes from Lemma \ref{cz}. Similarly,
\[
\left|\int_{K\setminus 3I_{|n|}(y)} f^0(z)\cot\left(\frac{y-z}{2}\right) dz\right|  \le C\lambda n|K|.
\]
Thus,
\[
A^1 \le \sum_{(J,K)\in \mathcal{G}} \int_{\gamma F_{\epsilon}\cap\gamma J \cap \gamma K} C\lambda^2 n^2 |J| |K| |\sigma_m 1_{\overline{\beta F_{\epsilon}}}(y)|^2 dy.
\]
Let  $y\in \gamma F_{\epsilon} \cap\gamma J \cap \gamma K$, where $(J,K)\in \mathcal{G}$. Then, there exists an $I\in \F{\epsilon}$ such that
$y\in \gamma I\cap\gamma J \cap \gamma K $, where $(I,J,K)\in \mathcal{G}_{\epsilon} = \F{\epsilon}\times \F{}\times \F{}$. We give an upper bound for
the nonnegative real number $\sigma_m 1_{\overline{\beta F_{\epsilon}}}(y)$.

\begin{equation}\label{sigma<16/mI}
\begin{split}
    &  \sigma_m 1_{\overline{\beta F_{\epsilon}}}(y) = \frac{1}{\pi}\int_{-\pi}^{\pi}  1_{\overline{\beta F_{\epsilon}}}(x) K_m(y-x) dx \\
     & = \frac{1}{\pi}\int_{\overline{\beta F_{\epsilon}}} K_m(y-x) dx \le \frac{8}{m}\int_{\overline{\beta F_{\epsilon}}}\frac{1}{|y-x  \, (\Mod T)|^2}dx  \\
     & \le \frac{16}{m}\int_{\set{z: z> (\beta-\gamma)|I|/2}}\frac{1}{z^2} dz \le \frac{32}{\beta-\gamma}\frac{1}{m |I|}
\end{split}
\end{equation}
because for $y\in \gamma I, I\in \F{\epsilon}$ and $x\in \overline{\beta F_{\epsilon}}$ we have $x\notin \beta I$ and consequently
$|y-x \, (\Mod T)|> (\beta-\gamma)|I|/2$. Remark that $y-x (\Mod T)$ means  $y-x (\Mod T)= y-x + u2\pi \in T$ for a $u \in\{-1,0,1\}$. That is, if $y-x$ is not in interval $T$, then it is shifted by $2\pi$.

Moreover, if $|J|\ge |I|$, then by $I\in \F{\epsilon}$ which is equivalent with $|I|> \epsilon$ (and of course $I\in \F{}$) we also have
$|J|> \epsilon$, that is, $J\in \F{\epsilon}, y\in \gamma J$ and consequently by the fact that $x\in \overline{\beta F_{\epsilon}}$
\[
\sigma_m 1_{\overline{\beta F_{\epsilon}}}(y)  \le \frac{32}{\beta-\gamma}\frac{1}{m |J|}.
\]
The same can be said about $K$ because $y\in \gamma K$. This yields

\begin{equation}\label{sigma<16/mIJK}
\sigma_m 1_{\overline{\beta F_{\epsilon}}}(y)\le  \frac{32}{(\beta-\gamma)m \max(|I|, |J|, |K|)} \quad \mbox{for\ } y\in \gamma I\cap\gamma J\cap \gamma K.
\end{equation}

Henceforth it is easy to give a bound for $A^1$:
\[
\begin{split}
    &
A^1 \le \sum_{(J,K)\in \mathcal{G}}\frac{C\lambda^2 n^2}{(\beta-\gamma)^2m^2}\int_{\gamma F_{\epsilon}\cap\gamma J\cap \gamma K}
\frac{|J| |K|}{(|J| + |K|)^2} dy\\
& \le \frac{C\lambda^2}{(\beta-\gamma)^2}\sum_{(J,K)\in \mathcal{G}}\frac{|J| |K| |\gamma F_{\epsilon}\cap\gamma J\cap \gamma K|}{(|J| + |K|)^2} \\
& \le \frac{C\lambda^2}{(\beta-\gamma)^2}\sum_{(J,K)\in \F{}\times\F{}}\frac{|J| |K| |\gamma J\cap \gamma K|}{(|J| + |K|)^2}\\
& \le \frac{C\lambda^2}{(\beta-\gamma)^2}\sum_{(J,K)\in \F{}\times\F{}} |\gamma J\cap \gamma K|
=: A^{1,1}.
\end{split}
\]
Thus, by Lemma \ref{gammaIgammaJ} and Lemma \ref{cz} we have
\[
A^1 \le A^{1,1} \le \frac{C\lambda^2}{(\beta-\gamma)^2}C_{\gamma} |F|
\le C_{\beta,\gamma}\|f\|_1\lambda.
\]
Next, turn our attention to $A^2$. Since $x\notin 3I_{|n|}(y)$, then as in the investigation of $A^1$ we have
\[
\left|\int_{J\setminus 3I_{|n|}(y)} f^0(x)\cot\left(\frac{y-x}{2}\right) dx\right|  \le C\lambda n|J|
\]
again. Meanwhile, as in the investigation of $A^1$ for any $y\in \gamma I\cap \gamma \tilde J$ ($I\in \F{\epsilon}, \tilde J\in \F{}$) it yields
\[
\sigma_m 1_{\overline{\beta F_{\epsilon}}}(y)\le  \frac{32}{m(\beta-\gamma)}\frac{1}{\max(|I|, |\tilde J|)}
\]
and
\[
\left|\sigma_m 1_{\overline{\beta F_{\epsilon}}}(y)\right|^2\le  \frac{32}{m(\beta-\gamma)}\frac{1}{\max(|I|, |\tilde J|)} \le \frac{32}{m(\beta-\gamma)}\frac{1}{|\tilde J|}
\]
($B\le a,C$ gives $B^2\le Ca$.)
Recall that $\gamma J = \bigcup_{j=1/2-\gamma/2}^{\gamma/2-1/2}J^{(j)}$ and fix a $j$.

Let $\F{}^{\prime}$ be a subset of $\F{}$. We call the dyadic interval $\tilde J^{(j)}$  maximal with respect to $\F{}^{\prime}$ if
$\tilde J\in \F{}^{\prime}$ and there is no interval
$ J\in \F{}^{\prime}$ such that $\tilde J^{(j)} \subsetneq J^{(j)}$.
This will be abbreviated as $\tilde J\in \F{}^{\prime}, \tilde J^{(j)}$ maximal.
If $\F{}^{\prime}= \F{}$, then we simply use the term maximal intervals.
Thus, recalling that two dyadic intervals are disjoint or one of them contains the other, one has
\[
\begin{split}
    & A^2 \le \sum_{j=1/2-\gamma/2}^{\gamma/2-1/2}\sum_{K\in \F{}}\sum_{\tilde J\in\F{}, \tilde J^{(j)} \mbox{\ {\tiny maximal}}}\sum_{J\in\F{},  J^{(j)} \subset\tilde J^{(j)}}
    \int_{(\gamma F_{\epsilon}\cap J^{(j)}) \setminus \gamma K}
    C\lambda n|J|    \\
& \times \left|\int_{K\setminus 3I_{|n|}(y)} f^0(z)\cot\left(\frac{y-z}{2}\right) dz\right| \left|\sigma_m 1_{\overline{\beta F_{\epsilon}}}(y)\right|^2 dy \\
& \le
 \sum_{j=1/2-\gamma/2}^{\gamma/2-1/2}\sum_{K\in \F{}}\sum_{\tilde J\in\F{}, \tilde J^{(j)} \mbox{\ {\tiny maximal}}}\sum_{J\in\F{},  J^{(j)} \subset\tilde J^{(j)}}
    \int_{(\gamma F_{\epsilon}\cap \tilde J^{(j)}) \setminus \gamma K}
    C\lambda n|J|    \\
& \times \left|\int_{K\setminus 3I_{|n|}(y)} f^0(z)\cot\left(\frac{y-z}{2}\right) dz\right| \left|\sigma_m 1_{\overline{\beta F_{\epsilon}}}(y)\right|^2 dy \\
& \le
\sum_{j=1/2-\gamma/2}^{\gamma/2-1/2}\sum_{K\in \F{}}\sum_{\tilde J\in\F{}, \tilde J^{(j)} \mbox{\ {\tiny maximal}}}\sum_{J\in\F{},  J^{(j)} \subset\tilde J^{(j)}}
    \int_{(\gamma F_{\epsilon}\cap \tilde J^{(j)}) \setminus \gamma K}
    C\lambda n|J|  \frac{C}{m(\beta-\gamma)}\frac{1}{|\tilde J|}  \\
& \times \left|\int_{K\setminus 3I_{|n|}(y)} f^0(z)\cot\left(\frac{y-z}{2}\right) dz\right|  dy. \\
 \end{split}
\]
Since by Lemma \ref{gammaIgammaJ}
\[
\sum_{J\in\F{},  J^{(j)} \subset\tilde J^{(j)}} \frac{|J|}{|\tilde J|} = \sum_{J\in\F{},  J^{(j)} \subset\tilde J^{(j)}} \frac{|J^{(j)}|}{|\tilde J^{(j)}|} \le C_{\gamma},
\]
in the estimation of $A^2$ we have
\[
\begin{split}
    & A^2 \le \frac{C_{\gamma}\lambda}{\beta-\gamma}
\sum_{j=1/2-\gamma/2}^{\gamma/2-1/2}\sum_{K\in \F{}}\sum_{\tilde J\in\F{}, \tilde J^{(j)} \mbox{\ {\tiny maximal}}}
    \int_{(\gamma F_{\epsilon}\cap \tilde J^{(j)}) \setminus \gamma K}
   \\
& \times \left|\int_{K\setminus 3I_{|n|}(y)} f^0(z)\cot\left(\frac{y-z}{2}\right) dz\right|  dy.
 \end{split}
\]

Since the maximal dyadic intervals   $\tilde J^{(j)}$ $(\tilde J\in\F{})$ are disjoint, then
\[
\begin{split}
    & A^2 \le \frac{C_{\gamma}\lambda}{\beta-\gamma}\sum_{K\in \F{}}
    \int_{T \setminus \gamma K}
\left|\int_{K\setminus 3I_{|n|}(y)} f^0(z)\cot\left(\frac{y-z}{2}\right) dz\right|  dy.
 \end{split}
\]

Two dyadic intervals are disjoint or one of them is contained in the other. Is it possible that
$I_{|n|}(y) \subset K$? If yes, then $y\in I_{|n|}(y) \subset K\subset 3K \subset \gamma K$ gives a contradiction since  $y\in T\setminus \gamma K$.
That is, either $K\setminus I_{|n|}(y)=\emptyset$ or $K\setminus I_{|n|}(y) = K$.
On the other hand, if $I^+_{|n|}(y) \subset K$, then $y\in I_{|n|}(y) 
\subset 3K \subset\gamma K$ gives the same contradiction. That is,
either $K\setminus I^+_{|n|}(y)=\emptyset$ or $K\setminus I^+_{|n|}(y) = K$. The same situation concerns  $I^-_{|n|}(y)$ which means
that $K\setminus 3I_{|n|}(y)$ is either  $\emptyset$ or $K$.

In the first case ($K\setminus 3I_{|n|}(y)=\emptyset$) there is nothing to prove. In the second case
($K\setminus 3I_{|n|}(y)=K$)
we have $\int_{K\setminus 3I_{|n|}(y)} f^0(z) dz = 0$, because by Lemma \ref{cz} $\int_{K} f^0(z) dz = 0$.
Then, denoting by $z_0$ the center of the interval $K$ we have
\begin{equation}\label{intK-3I}
\begin{split}
   &
 \left|\int_{K\setminus 3I_{|n|}(y)} f^0(z)\cot\left(\frac{y-z}{2}\right) dz\right| =
 \left|\int_{K\setminus 3I_{|n|}(y)} f^0(z)\left(\cot\left(\frac{y-z}{2}\right)- \cot\left(\frac{y-z_0}{2}\right)\right) dz\right|\\
  & \le C \int_{K\setminus 3I_{|n|}(y)} |f^0(z)| \frac{|K|}{\sin^2\left(\frac{y-z_0}{2}\right)} dz \\
  & \le
      \frac{C|K|}{\sin^2\left(\frac{y-z_0}{2}\right)}\int_{K}|f^0(z)| dz
\end{split}
\end{equation}
because
\[
\left|\cot\left(\frac{y-z}{2}\right)- \cot\left(\frac{y-z_0}{2}\right)\right| \le C|z-z_0|\frac{1}{\sin^2\left(\frac{y-z_0}{2}\right)}
\le C\frac{|K|}{\sin^2\left(\frac{y-z_0}{2}\right)}.
\]
Besides, as above we have $\int_{K}|f^0(z)| dz \le 4\lambda |K|$. Thus,

\[
A^2 \le \frac{C_{\gamma}\lambda^2}{\beta-\gamma}\sum_{K\in \F{}}|K|^2
\int_{T \setminus \gamma K}\frac{1}{\sin^2\left(\frac{y-z_0}{2}\right)} dy.
\]

Consequently,
\[
A^2 \le \frac{C_{\gamma}\lambda^2}{\beta-\gamma}\sum_{K\in \F{}}|K|^2
\int_{T \setminus \gamma K}\frac{1}{\left(y-z_0\right)^2} dy.
\]
Since $z_0$ is the center of $K$, we have ($\gamma > 5$)
\[
\int_{T \setminus \gamma K}\frac{1}{\left(y-z_0\right)^2} dy \le
C\int_{\frac{\gamma-1}{2}|K|}^{\infty}\frac{1}{t^2} dt \le \frac{C}{|K|}.
\]
Consequently, we have
\[
\begin{split}
    &
A^2 \le \frac{C_{\gamma}\lambda^2}{\beta-\gamma}\sum_{K\in \F{}}|K|^2
\frac{1}{|K|} \le \frac{C_{\gamma}\lambda^2}{\beta-\gamma}\sum_{K\in \F{}}|K|\\
& = \frac{C_{\gamma}\lambda^2}{\beta-\gamma}|F| \le \frac{C_{\gamma}}{\beta-\gamma}\|f\|_1\lambda \le C_{\beta,\gamma}\|f\|_1\lambda.
\end{split}
\]
This completes the discussion for $A^2$. The sum $A^3$ is similar. Only the role of $J$ and $K$ is changed and therefore we also have
\[
A^3 \le C_{\beta,\gamma}\|f\|_1\lambda.
\]
Finally, we turn our attention to the sum $A^4$.
Apply the inequality $\sigma_m 1_{\overline{\beta F_{\epsilon}}}(y)\le 1$ and then
\[
\begin{split}
   &
A^4 \\
& \le \sum_{J, K\in \F{}}\int_{T \setminus (\gamma J \cup \gamma K)}
\left|\int_{J\setminus 3I_{|n|}(y)} f^0(x)\cot\left(\frac{y-x}{2}\right) dx\right|
  \left|\int_{K\setminus 3I_{|n|}(y)} f^0(z)\cot\left(\frac{y-z}{2}\right) dz\right| dy.
\end{split}
\]
In the same way as in the investigation of $A^2$ we get again ($\gamma > 5$) (see (\ref{intK-3I}))
\[
\begin{split}
   &
 \left|\int_{K\setminus 3I_{|n|}(y)} f^0(z)\cot\left(\frac{y-z}{2}\right) dz\right| 
    \le   \frac{C|K|}{\sin^2\left(\frac{y-z_0}{2}\right)}\int_{K}|f^0(z)| dz
   \le   \frac{C|K|^2\lambda}{\sin^2\left(\frac{y-z_0}{2}\right)},
\end{split}
\]
where $z_0$ is the center of $K$. Similarly,
\[
 \left|\int_{J\setminus 3I_{|n|}(y)} f^0(x)\cot\left(\frac{y-x}{2}\right) dx\right| \le  \frac{C|J|^2\lambda}{\sin^2\left(\frac{y-x_0}{2}\right)},
\]
where $x_0$ is the center of $J$.
Thus,
\begin{equation}\label{A4estimation}
 A^4
 \le C\lambda^2 \sum_{J, K\in \F{}}\int_{T \setminus (\gamma J \cup \gamma K)}
\frac{|J|^2 |K|^2}{\sin^2\left(\frac{y-x_0}{2}\right) \sin^2\left(\frac{y-z_0}{2}\right)} dy.
\end{equation}

Let $A^{4,1}$ be the part of the right hand side of (\ref{A4estimation}) for which $J\not = K$ (then $y_0\not= z_0$).
In this case
apply the inequality $\sin^2(a-b)\le \sin^2a + \sin^2b + 2|\sin a \sin b|\le 2(\sin^2a + \sin ^2b)$ for $a=(y-z_0)/2, b=(y-x_0)/2$.
We have
\begin{equation}
\label{1/(y-x0)(y-z0)split_to_sums}
\frac{1}{\sin^2\left(\frac{y-x_0}{2}\right) \sin^2\left(\frac{y-z_0}{2}\right)} \le
\frac{2}{\sin^2\left(\frac{x_0-z_0}{2}\right)}\left(
\frac{1}{\sin^2\left(\frac{y-x_0}{2}\right)} +
\frac{1}{\sin^2\left(\frac{y-z_0}{2}\right)}
\right).
\end{equation}
Just as in the investigation of the sum $A^2$ we get again ($\gamma > 5$)
\[
\begin{split}
& \int_{T \setminus (\gamma J \cup \gamma K)}
\frac{1}{\sin^2\left(\frac{y-x_0}{2}\right)} + \frac{1}{\sin^2\left(\frac{y-z_0}{2}\right)} dy \\
& \le
C\int_{T\setminus \gamma J}\frac{1}{(y-x_0)^2} dy + C\int_{T\setminus \gamma K}\frac{1}{(y-z_0)^2} dy \\
& \le C\left(\frac{1}{|J|} + \frac{1}{|K|}\right).
\end{split}
\]
Consequently,
\[
A^{4,1} \le C\lambda^2\sum_{J,K\in \F{}, J\not= K} \frac{|J|^ 2|K|^2}{\sin^2\left(\frac{x_0-z_0}{2}\right)}\left(\frac{1}{|J|} + \frac{1}{|K|}\right).
\]
$x_0$ and $z_0$ are the centers of the disjoint dyadic intervals $J$ and $K$. Therefore,
\[
\sum_{J\in \F{}, J\not= K} \frac{|J|}{\sin^2\left(\frac{x_0-z_0}{2}\right)}
\le C\int_{T\setminus K}\frac{1}{(t-z_0)^2} dt \le C \int_{\frac{|K|}{2}}^{+\infty}\frac{1}{t^2} dt \le \frac{C}{|K|}.
\]
Similarly,
\[
\sum_{K\in \F{}, K\not= J} \frac{|K|}{\sin^2\left(\frac{x_0-z_0}{2}\right)} \le \frac{C}{|J|}.
\]
These assumptions give
\[
A^{4,1} \le C\lambda^2\sum_{K\in \F{}}|K| +  C\lambda^2\sum_{J\in \F{}}|J| \le C\lambda^2 |F| \le C\|f\|_1\lambda.
\]
On the other hand, let
$A^{4,2}$ be the part in the right hand side of (\ref{A4estimation}) (estimation of $A^4$) for which $J=K$. For this we have
\[
\begin{split}
&
A^{4,2} \\
& \le C\lambda^2 \sum_{J\in \F{}}\int_{T \setminus \gamma J }
\frac{|J|^4}{\sin^4\left(\frac{y-x_0}{2}\right)} dy \\
& \le
C\lambda^2 \sum_{J\in \F{}}|J|^4\int_{T \setminus \gamma J }\frac{1}{(y-x_0)^4} dy \le C\lambda^2 \sum_{J\in \F{}}|J|
\le  C\|f\|_1\lambda.
\end{split}
\]
Consequently,
\begin{equation}\label{A4estimation2}
A^4 \le C\|f\|_1\lambda.
\end{equation}

This completes the proof of Lemma \ref{gammaFjH2}.
\end{proof}

The second lemma to be proved is
\begin{lem}\label{gammaFjS}
Let $\beta > \gamma >5$ be odd integers, $\epsilon>0$, $l,m\in\N$ with $l\le 100 m$ and $f\in L^1(T), \lambda >\|f\|_1/(2\pi)$. Then the inequality
\[
\int_{\gamma F_{\epsilon}}|S_l f(y)|^2 |\sigma_m 1_{\overline{\beta F_{\epsilon}}}(y)|^2 dy \le C_{\beta, \gamma}\|f\|_1\lambda,
\]
holds, where the constant $C_{\beta, \gamma}$ can depend only on $\beta$ and $\gamma$ (and it is uniform in $f, l, m, \epsilon, \lambda$).
\end{lem}
\begin{proof}
Denote by $|l|$ the lower integer part of the binary logarithm of $l$. It is well-known that

\[
\frac{1}{e^{\imath z}-1} = -\frac{1}{2} - \frac{\imath}{2}\cot \frac{z}{2}
\]

and for the Dirichlet kernel

\begin{equation}\label{dirichlet_kernel}
D_l(z) = \frac{1}{2}\sum_{k=-l}^{l}e^{\imath kz} = \left(e^{\imath(l+1)z} - e^{-\imath lz}\right)\left(-\frac{1}{4} - \frac{\imath}{4}\cot \frac{z}{2}\right).
\end{equation}

Then, let

\begin{equation}\label{modified_partial}
\tilde S_l f(y) := \frac{1}{\pi}\int_{T\setminus 3I_{|l|}(y)} f(x)D_l(y-x) dx.
\end{equation}

From the definition of the Hilbert transform (\ref{hilbert_def}) we have
\[
\begin{split}
    &  |\tilde S_l f(y) | \le \frac{1}{2\pi}\int_{T\setminus 3I_{|l|}(y)}|f(x)| dx \\
     & + \frac{1}{4\pi}\left|H_l(f(\cdot)e^{-\imath(l+1)\cdot})(y)\right| +  \frac{1}{4\pi}\left|H_l(f(\cdot)e^{\imath l\cdot})(y)\right|
\end{split}
\]
and by this we also have
\begin{equation}\label{^2estimation_on_modified_partial}
|\tilde S_l f(y) |^2 \le \|f\|_1^2 + \left|H_l(f(\cdot)e^{-\imath(l+1)\cdot})(y)\right|^2 + \left|H_l(f(\cdot)e^{\imath l\cdot})(y)\right|^2.
\end{equation}

Recall the Calderon-Zygmund decomposition lemma, that is, Lemma \ref{cz}. Then $\|f\|_1^2 \le 2\pi\|f\|_1\lambda$. Besides, since
$|f(\cdot)| = |f(\cdot)e^{\imath k\cdot}|$ for any $k\in \mathbb{Z}$, the set $\F{}$ for the function $f(\cdot)e^{-\imath k\cdot}$ will be the same as for the
function $f$. That is, by Lemma \ref{gammaFjH2} we have
\[
\int_{\gamma F_{\epsilon}}|\tilde S_l f(y)|^2 |\sigma_m 1_{\overline{\beta F_{\epsilon}}}(y)|^2 dy \le C_{\beta, \gamma}\|f\|_1\lambda.
\]
Define  the operator $E_l$ as follows:
\begin{equation}\label{Eldef}
E_lf(y) := l\int_{3I_{|l|}(y)}f(x) dx.
\end{equation}

Now we have to check the difference of $S_lf$ and $\tilde S_lf$. This is nothing else but
\begin{equation}\label{SltildeSldifference}
\frac{1}{\pi}\left|\int_{3I_{|l|}(y)}f(x)D_l(y-x) dx\right| \le \frac{l+1/2}{\pi}\int_{3I_{|l|}(y)}|f(x)| dx \le E_l|f|(y).
\end{equation}

We investigate the operator $E_l$ as it would be in the statement of Lemma \ref{gammaFjS} instead of $S_l$. Recall again the Calderon-Zygmund decomposition lemma, that is, Lemma \ref{cz} for $f=\sum_{i=0}^{\infty}f_i = f_0+f^0$. For $f_0 $ we have
$|E_lf_0(y)|^2 \le 12\pi \|f_0\|_{\infty} |E_lf_0(y)|$ and consequently,
\[
\int_{\gamma F_{\epsilon}}|E_l f_0(y)|^2 |\sigma_m 1_{\overline{\beta F_{\epsilon}}}(y)|^2 dy \le C\int_{-\pi}^{\pi}\|f_0\|_{\infty}|E_lf_0(y)| dy \le
C\|f_0\|_{\infty} \|f_0\|_1 \le C\|f\|_1\lambda.
\]
Next, let's see the discussion for $f^0$.
Using the notation of the proof of Lemma \ref{gammaFjH2} we have

\[
\begin{split}
     & \int_{\gamma F_{\epsilon}}|E_l f^0(y)|^2 |\sigma_m 1_{\overline{\beta F_{\epsilon}}}(y)|^2 dy\\
& \le  \sum_{J,K\in \F{}}\int_{\gamma F_{\epsilon}}l^2
\left|\int_{J\cap 3I_{|l|}(y)} f^0(x) dx\right|
\left|\int_{K\cap 3I_{|l|}(y)} f^0(z) dz\right|
  |\sigma_m 1_{\overline{\beta F_{\epsilon}}}(y)|^2 dy.
\end{split}
\]
The sum over pairs $(J,K)\in  \F{} \times \F{}$ will be divided into the following parts: ($\mathcal{G} =  \F{}\times \F{}$)
\[
\begin{split}
      &  \sum_{(J,K)\in \mathcal{G}} \int_{\gamma F_{\epsilon}} = \sum_{(J,K)\in \mathcal{G}} \int_{\gamma F_{\epsilon}\cap\gamma J \cap \gamma K} \\
& +  \sum_{(J,K)\in \mathcal{G}} \int_{(\gamma F_{\epsilon}\cap\gamma J) \setminus \gamma K}
 +  \sum_{(J,K)\in \mathcal{G}} \int_{(\gamma F_{\epsilon}\cap\gamma K) \setminus \gamma J}\\
 & +  \sum_{(J,K)\in \mathcal{G}} \int_{\gamma F_{\epsilon}\setminus(\gamma J\cup \gamma K)}
 =: \sum_{i=1}^{4}A^i.
\end{split}
\]
First, check $A^3$. In this case $y\in \gamma K \setminus \gamma J$. That is, $y\notin \gamma J$.

If $|I_{|l|}(y)| =\frac{2\pi }{2^{|l|}} < |J|$, then
(by $y\in I_{|l|}(y)$) we have $3I_{|l|}(y)\cap J =\emptyset$ and consequently
$\int_{J\cap 3I_{|l|}(y)} f^0(x) dx=0$ gives that the every addend in $A^3$ corresponding to intervals $J$ of this type is $0$.

On the other hand,
$|I_{|l|}(y)|  \ge  |J|$ gives that either $J\cap I_{|l|}(y)=\emptyset$ or $J\cap I_{|l|}(y)=J$. The same can be said about
the intervals $I^+_{|l|}(y)$ and $I^-_{|l|}(y)$.
This gives that either $J\cap 3I_{|l|}(y)= (J\cap I_{|l|}(y))\cup (J\cap I^+_{|l|}(y))\cup(J\cap I^-_{|l|}(y)) = \emptyset$ or
$J\cap 3I_{|l|}(y)= J$. In both cases by Lemma \ref{cz} we have $\int_{J\cap 3I_{|l|}(y)}f^0(x) dx =0$.

That is, in any cases $y\notin\gamma J$ gives $\int_{J\cap 3I_{|l|}(y)}f^0(x) dx =0$. This implies
that every addend in $A^3$ corresponding to any interval $J$ (regardless of its measure) is $0$.
Thus, $A^3=0$. The same argument gives  $A^4=0$ and changing
the role of $J$ and $K$ gives  $A^2=0$. That is, the only sum (or you may say case) remained to be investigated is $A^1$. In this situation we can just follow the corresponding steps of the proof of Lemma \ref{gammaFjH2} (see (\ref{sigma<16/mI}) and (\ref{sigma<16/mIJK})).
That is, again we have for $y\in \gamma I\cap \gamma J\cap\gamma K$ (where $I$ is some element of $\F{\epsilon}$):
\[
\sigma_m 1_{\overline{\beta F_{\epsilon}}}(y)  \le \frac{C}{\beta-\gamma}\frac{1}{m (|I|+|J|+|K|)}\le \frac{C}{\beta-\gamma}\frac{1}{m (|J|+|K|)}.
\]
Denoting $\F{}\times\F{}$ by $\mathcal{G}$ again
\[
\begin{split}
&
A^1 \le C\sum_{(J,K)\in \mathcal{G}} \int_{\gamma F_{\epsilon}\cap\gamma J \cap \gamma K}
l^2\left|\int_{J\cap 3I_{|l|}(y)} f^0(x) dx\right| \left|\int_{K\cap 3I_{|l|}(y)} f^0(z) dz\right| \\
& \times
\frac{1}{(\beta-\gamma)^2m^2(|J|+|K|)^2}.
\end{split}
\]
By the decomposition Lemma \ref{cz} we have
\[
\left|\int_{J\cap 3I_{|l|}(y)} f^0(x) dx\right| \le |J| |J|^{-1}\int_{J}|f^0|\le 4\lambda|J|, \quad \left|\int_{K\cap 3I_{|l|}(y)} f^0(z) dz\right| \le 4\lambda|K|.
\]
Thus,
\[
\begin{split}
    &
A^1 \le \sum_{(J,K)\in \mathcal{G}}\frac{C\lambda^2 l^2}{(\beta-\gamma)^2m^2}\int_{F_{\epsilon}\cap\gamma J\cap \gamma K}
\frac{|J| |K|}{(|J| + |K|)^2} dy\\
& \le \frac{C\lambda^2}{(\beta-\gamma)^2}\sum_{(J,K)\in \mathcal{G}}\frac{|J| |K| |F_{\epsilon}\cap\gamma J\cap \gamma K|}{(|J| + |K|)^2} \\
& \le \frac{C\lambda^2}{(\beta-\gamma)^2}\sum_{(J,K)\in \F{}\times\F{}} |\gamma J\cap \gamma K|
=: A^{1,1}.
\end{split}
\]
Have a look at Lemma \ref{gammaIgammaJ} or alternatively recall that in the proof of Lemma \ref{gammaFjH2} it was proved that $A^{1,1}$ is not greater than $\frac{C_{\gamma}}{(\beta-\gamma)^2}\|f\|_1\lambda$.
That is, for the operator $E_l $ we have
\[
\int_{\gamma F_{\epsilon}}|E_l f^0(y)|^2 |\sigma_m 1_{\overline{\beta F_{\epsilon}}}(y)|^2 dy \le C_{\beta,\gamma}\|f\|_1\lambda.
\]
Since this inequality is also proved for the function $f_0$, then it is also verified for $f=f_0+f^0$. Apply this inequality for the function $|f|$ instead of $f$.
(Remark that $\F{}$ depends only on $|f|$ and $\lambda$.)
This gives
\[
\int_{\gamma F_{\epsilon}}|E_l|f|(y)|^2 |\sigma_m 1_{\overline{\beta F_{\epsilon}}}(y)|^2 dy \le C_{\beta,\gamma}\|f\|_1\lambda.
\]

That is, by
$\left|\int_{3I_{|l|}(y)}f(x)D_l(y-x) dx\right| \le  \pi E_l|f|(y)$ we have
\[
\int_{\gamma F_{\epsilon}}\left|\int_{3I_{|l|}(y)}f(x)D_l(y-x) dx\right|^2
\sigma_m 1_{\overline{\beta F_{\epsilon}}}(y)|^2 dy \le C_{\beta,\gamma}\|f\|_1\lambda
\]
and finally taking into account that $\tilde S_l$ has already been estimated it follows that
\[
\int_{\gamma F_{\epsilon}}\left|S_lf(y)\right|^2
\sigma_m 1_{\overline{\beta F_{\epsilon}}}(y)|^2 dy \le C_{\beta,\gamma}\|f\|_1\lambda.
\]
This completes the proof of Lemma \ref{gammaFjS}.
\end{proof}

By Lemma \ref{gammaFjS} it is easy to prove the next corollary concerning the difference of partial sums and  de la Vall\'ee-Poussin means. Namely,
\begin{cor}\label{gammaFjS-V}
Let $\beta > \gamma >5$ be odd integers, $\epsilon>0$, $n,m\in\N$ with $n\le 50m$ and $f\in L^1(T), \lambda >\|f\|_1/(2\pi)$. Then the inequality
\[
\int_{\gamma F_{\epsilon}}|S_nf(y)-V_nf(y)|^2 |\sigma_m 1_{\overline{\beta F_{\epsilon}}}(y)|^2 dy \le C_{\beta, \gamma}\|f\|_1\lambda
\]
holds. The constant $C_{\beta, \gamma}$ can depend only on $\beta$ and $\gamma$ (and it is uniform in $f, n, m, \epsilon, \lambda$).
\end{cor}

\begin{proof}
  The equality $V_nf = \frac{1}{n}\sum_{l=n}^{2n-1}S_lf$ and Lemma \ref{gammaFjS} give
\[
\begin{split}
    & \int_{\gamma F_{\epsilon}}|V_nf(y)|^2 |\sigma_m 1_{\overline{\beta F_{\epsilon}}}(y)|^2 dy \\
    & \le
\frac{1}{n}\sum_{l=n}^{2n-1}\int_{\gamma F_{\epsilon}}|S_lf(y)|^2 |\sigma_m 1_{\overline{\beta F_{\epsilon}}}(y)|^2 dy\\
& \le \frac{1}{n}\sum_{l=n}^{2n-1} C_{\beta, \gamma}\|f\|_1\lambda \\
& \le  C_{\beta, \gamma}\|f\|_1\lambda.
\end{split}
\]
 We also used the well-known inequality $\left|\frac{1}{n}\sum_{i=1}^{n} x_i\right|^2 \le \frac{1}{n}\sum_{i=1}^{n} |x_i|^2$ for complex numbers. Then by Lemma \ref{gammaFjS} the proof of Corollary \ref{gammaFjS-V} is complete.
\end{proof}

\section{The sum of integrals on $\gamma F\setminus\gamma F_{\beta_j}$}
This section is  probably the most difficult part of this paper. But its understanding is helped by the reading of the  previous  section.
Similar methods and notation are used in this section.

Throughout  this section let $(n_j)$ be a lacunary sequence of natural numbers. More precisely, $n_{j+1}/n_j \ge 2$ for each $j\in \mathbb{N}$.
Set the sequence $(\beta_j)$ as
$n_j\beta_j = 20(j+1)\log^2(j+1)$ ($j\in\mathbb{N}$) (thus $n_j\beta_j>16$). Let $f\in L^1$ and use the notation of Lemma \ref{cz}.

\begin{lem}\label{gammaF-gammaFjH}
Let $\gamma >5$ be an odd integer,  $N\in\N$. Let $f, g_j\in L^1(T)$ be such that $|g_j|=|f|$ everywhere   for $j=1,\dots, N, \lambda >\|f\|_1/(2\pi)$.
Then the inequality
\[
\sum_{j=1}^{N}\int_{\gamma F\setminus\gamma F_{\beta_j}}|H_{n_j} g_j(y)|^2 dy \le C_{\gamma} N\log^5 (N+1)\|f\|_1\lambda
\]
holds. The constant $C_{\gamma}$ can depend only on $\gamma$ (and it is uniform in $f, (g_j), (n_j), N, \lambda$).
\end{lem}
\begin{proof}
Without loss of generality we can suppose that $f$ is real.  First of all, $N\ge 32$ can be supposed because in the case of $N<32$ we can complement
$n_1, \dots , n_N$ with $n_{N+1}\ge 2n_N, n_{N+2}\ge 2n_{N+1}, \dots, n_{32} \ge  2n_{31}$ and the left hand side of the statement of Lemma \ref{gammaF-gammaFjH}
is increased, while the right hand side is still a constant (depending on $\gamma$) multiplied by $\|f\|_1\lambda$.
Apply Lemma \ref{cz}, that is the Calderon-Zygmund decomposition for the function $g_j$. That is, $g_j = \sum_{i=0}^{\infty}g_{j,i} = g_{j,0}+g_j^0$.
  Since $|g_j|=|f|$ everywhere, the set of intervals $\F{}$ for the function $g_j$ and $f$ are the same.
  Since the Hilbert transform is
  of type $(L^2, L^2)$, one has
\[
\int_{\gamma F\setminus\gamma F_{\beta_j}}|H_{n_j} g_{j,0}(y)|^2 dy \le \|H_{n_j} g_{j,0}\|_2^2 \le C\|g_{j,0}\|_2^2 \le C\|g_{j,0}\|_1\lambda  \le C\|g_{j}\|_1\lambda
\le C\|f\|_1\lambda.
\]
That is, instead of the functions $g_j$ it is enough to investigate $g_j^0$ only.
\[
\begin{split}
    &  \int_{\gamma F\setminus\gamma F_{\beta_j}}|H_{n_j} g_j^0(y)|^2 dy\\
     & = \sum_{J,K\in \F{}}\int_{\gamma F\setminus\gamma F_{\beta_j}}
     \int_{J\setminus 3I_{|n_j|}(y)} g_j^0(x)\cot\left(\frac{y-x}{2}\right) dx
\int_{K\setminus 3I_{|n_j|}(y)} g_j^0(z)\cot\left(\frac{y-z}{2}\right) dz dy.
\end{split}
\]
The sum over pairs $(J,K)\in \F{}\times\F{}$ and the integral $\int_{\gamma F\setminus\gamma F_{\beta_j}}$ will be divided into the following parts:
\[
\begin{split}
    & \sum_{J,K\in \F{}}\int_{\gamma F\setminus\gamma F_{\beta_j}} = \sum_{J,K\in \F{}}\int_{(\gamma F\setminus\gamma F_{\beta_j})\cap \gamma J\cap\gamma K} \\
&    + \sum_{J,K\in \F{}}\int_{((\gamma F\setminus\gamma F_{\beta_j})\cap \gamma J)\setminus\gamma K} \\
 &   + \sum_{J,K\in \F{}}\int_{((\gamma F\setminus\gamma F_{\beta_j})\cap \gamma K)\setminus\gamma J}\\
  &  + \sum_{J,K\in \F{}}\int_{(\gamma F\setminus\gamma F_{\beta_j})\setminus (\gamma J\cup\gamma K)} =:\sum_{i=1}^{4}A_j^i.
    \end{split}
    \]

    \textbf{First, we investigate $A_j^4$.}

    Is it possible that any of the dyadic intervals $I_{|n_j|}(y), I^+_{|n_j|}(y), I^-_{|n_j|}(y)$ is  a subset of $K$, when
    $y\in (\gamma F\setminus\gamma F_{\beta_j})\setminus (\gamma J\cup\gamma K)$? No, because $I^+_{|n_j|}(y)\subset K$ or  $I^-_{|n_j|}(y)\subset K$ would
    give  $y\in I_{|n_j|}(y)\subset \gamma K$ ($\gamma >3$) and this does not hold. Consequently, either
    $K\setminus 3I_{|n_j|}(y)=\emptyset$ or $K\setminus 3I_{|n_j|}(y)=K$. 
    In the first case ($K\setminus 3I_{|n_j|}(y)=\emptyset$) there is nothing to prove. In the second case
($K\setminus 3I_{|n_j|}(y)=K$)
\[
\int_{K\setminus 3I_{|n_j|}(y)} g_j^0(z) dz =0
\]
which gives (as in the proof of Lemma \ref{gammaFjH2}, at (\ref{intK-3I}))
\[
\begin{split}
   &
 \left|\int_{K\setminus 3I_{|n_j|}(y)} g_j^0(z)\cot\left(\frac{y-z}{2}\right) dz\right|
   \le \frac{C|K|}{\sin^2\left(\frac{y-z_0}{2}\right)}\int_{K}|g_j^0(z)| dz
   \le \frac{C\lambda |K|^2}{\sin^2\left(\frac{y-z_0}{2}\right)},
\end{split}
\]
where $z_0$ is the center of $K$.

\noindent
We can say the same with respect to the integral
$ \left|\int_{J\setminus 3I_{|n_j|}(y)} g_j^0(x)\cot\left(\frac{y-x}{2}\right) dx\right|$ and consequently
\[
A_j^4
 \le C\lambda^2 \sum_{ J, K\in \F{}}
 |J|^2 |K|^2
 \int_{(\gamma F\setminus\gamma F_{\beta_j}) \setminus (\gamma J \cup \gamma K)}
\frac{1}{\sin^2\left(\frac{y-x_0}{2}\right) \sin^2\left(\frac{y-z_0}{2}\right)} dy,
\]
where $x_0$ is the center of $J$.
The right hand side of this inequality is almost the same as in (\ref{A4estimation}). More precisely,
the integral in (\ref{A4estimation}) is greater (or the same) since the domain of the integral is larger (or the same).

Consequently, by (\ref{A4estimation2})
\[
\sum_{j=1}^{N}A^4_j \le  C N\|f\|_1\lambda.
\]

\textbf{Next, investigate $A_j^2$.}

We have to integrate with respect to $y$ on the set $((\gamma F\setminus\gamma F_{\beta_j})\cap \gamma J)\setminus\gamma K$. We divide the
set $\gamma F\setminus\gamma F_{\beta_j}$ into two disjoint subsets:
\[
\begin{split}
&
\gamma F\setminus\gamma F_{\beta_j} = \left(\bigcup_{I\in\F{}}\gamma I \setminus \bigcup_{I\in\F{}, |I|>\frac{16}{n_j}}\gamma I\right)
\cup
\left(\bigcup_{I\in\F{}, |I|> \frac{16}{n_j}}\gamma I \setminus \bigcup_{I\in\F{}, |I|>\beta_j}\gamma I\right)\\
& = (\gamma F\setminus\gamma F_{16/n_j}) \cup (\gamma F_{16/n_j}\setminus\gamma F_{\beta_j})\\
& =: (\gamma F\setminus\gamma F_{16/n_j}) \cup \Delta_j.
\end{split}
\]

Split $A_j^2$ with respect to the sets above as
\[
A_j^2 =  \sum_{J,K\in \F{}}\int_{((\gamma F\setminus\gamma F_{16/n_j})\cap \gamma J)\setminus\gamma K} + \sum_{J,K\in \F{}}\int_{(\Delta_j\cap \gamma J)\setminus\gamma K}
=: A_j^{2,1} + A_j^{2,2}.
\]
First, check $A_j^{2,1}$.
If $y\in (\gamma F\setminus\gamma F_{16/n_j})\cap \gamma J $, then $|J|\le \frac{16}{n_j}$ which gives by the definition of $g_j^0$ and Lemma \ref{cz} that
\[
\left|\int_{J\setminus 3I_{|n_j|}(y)} g_j^0(x)\cot\left(\frac{y-x}{2}\right) dx\right| \le Cn_j\int_{J}|g_j^0(x)| dx
\le C\lambda n_j|J|.
\]
And of course $n_j|J| \le 16$ (this will also be needed later). On the other hand, as in the case of $A_j^4$ we have again ($z_0$ is the center of $K$):
\[
 \left|\int_{K\setminus 3I_{|n_j|}(y)} g_j^0(z)\cot\left(\frac{y-z}{2}\right) dz\right|
\le \frac{C\lambda |K|^2}{\sin^2\left(\frac{y-z_0}{2}\right)}.
\]

Recall that  the dyadic interval $\tilde J^{(k)}$  is called $\F{}\setminus \F{16/n_j}$ maximal  if $\tilde J\in \F{}\setminus \F{16/n_j}$ and if there is no interval
$ J\in \F{}\setminus \F{16/n_j}$ such that $\tilde J^{(k)} \subsetneq J^{(k)}$. In the displayed formula below
$\tilde J^{(k)}$  maximal  means $\tilde J^{(k)}$ is  $\F{}\setminus \F{16/n_j}$ maximal.
Recall that two dyadic intervals are disjoint or one of them contains the other. Then by Lemma \ref{gammaIgammaJ} we have

\[
\begin{split}
    & A_j^{2,1} \le
    \sum_{k=1/2-\gamma/2}^{\gamma/2-1/2}\sum_{K\in \F{}}
    \sum_{\tilde J\in\F{}\setminus \F{16/n_j}, \tilde J^{(k)} \mbox{\ {\tiny maximal}}}
    \sum_{J\in\F{},  J^{(k)} \subset\tilde J^{(k)}}
    C\lambda n_j|J|\\
    & \times \int_{((\gamma F\setminus\gamma F_{16/n_j})\cap  J^{(k)})\setminus \gamma K} \frac{C\lambda |K|^2}{\sin^2\left(\frac{y-z_0}{2}\right)} dy\\
& \le
    \sum_{k=1/2-\gamma/2}^{\gamma/2-1/2}\sum_{K\in \F{}}\sum_{\tilde J\in\F{}\setminus \F{16/n_j}, \tilde J^{(k)} \mbox{\ {\tiny maximal}}}\sum_{J\in\F{},  J^{(k)} \subset\tilde J^{(k)}}
    C\lambda n_j|J|\\
    & \times \int_{((\gamma F\setminus\gamma F_{16/n_j})\cap  \tilde J^{(k)})\setminus \gamma K} \frac{C\lambda |K|^2}{\sin^2\left(\frac{y-z_0}{2}\right)} dy\\
& \le
    \sum_{k=1/2-\gamma/2}^{\gamma/2-1/2}\sum_{K\in \F{}}\sum_{\tilde J\in\F{}\setminus \F{16/n_j}, \tilde J^{(k)} \mbox{\ {\tiny maximal}}}
    C_{\gamma}\lambda n_j|\tilde J| \int_{((\gamma F\setminus\gamma F_{16/n_j})\cap  \tilde J^{(k)})\setminus \gamma K} \frac{C\lambda |K|^2}{\sin^2\left(\frac{y-z_0}{2}\right)} dy.\\
\end{split}
\]

Now, we use the fact that $n_j|\tilde J| \le 16$ and also that the maximal dyadic intervals (with respect to any fixed subset of $\F{}$) 
are disjoint. Thus,
\[
\begin{split}
    & A_j^{2,1} \le C_{\gamma}\sum_{K\in \F{}}\int_{(\gamma F\setminus\gamma F_{16/n_j})\setminus \gamma K}
\frac{C\lambda^2 |K|^2}{\sin^2\left(\frac{y-z_0}{2}\right)} dy\\
& \le C_{\gamma}\lambda^2\sum_{K\in \F{}} |K|^2\int_{T\setminus\gamma K}
\frac{1}{\sin^2\left(\frac{y-z_0}{2}\right)} dy\\
& \le C_{\gamma}\lambda^2\sum_{K\in \F{}} |K|^2\int_{|K|}^{\infty}\frac{1}{t^2} dt \\
& \le C_{\gamma}\|f\|_1\lambda.
\end{split}
\]

Now, turn our attention to $A_j^{2,2}$. That is,  check the integral with respect to $y$ on the set $\Delta_j$.
Since $y\in \Delta_j$, then $y\notin \bigcup_{I\in \F{}, |I|>\beta_j}\gamma I 
= \gamma F_{\beta_j}$ and $y\in \gamma J$ give $J\notin \F{\beta_j}$, that is,
$|J|\le \beta_j$ and consequently,
\[
\left|\int_{J\setminus 3I_{|n_j|}(y)} g_j^0(x)\cot\left(\frac{y-x}{2}\right) dx\right| \le C n_j|J| |J|^{-1}\int_{J}|g_j^0|
\le C\lambda n_j|J|.
\]
Later on, we will also use that $n_j|J| \le n_j\beta_j =  20(j+1)\log^2 (j+1) \le 40N\log^2(N+1)$.

Moreover, as at the beginning of the investigation of the case of $A_j^4$ we have again that
since $y\notin \gamma K$,  either
$K\setminus 3I_{|n_j|}(y)=\emptyset$ or $K\setminus 3I_{|n_j|}(y)=K$. In both cases
we have again (in the first case there is nothing to prove)
\[
 \left|\int_{K\setminus 3I_{|n_j|}(y)} g_j^0(z)\cot\left(\frac{y-z}{2}\right) dz\right|
 \le \frac{C\lambda |K|^2}{\sin^2\left(\frac{y-z_0}{2}\right)}
\]
($z_0$ is the center of $K$). So, we have

\[
\begin{split}
   &
\int_{(\Delta_j\cap \gamma J)\setminus \gamma K} \left|\int_{J\setminus 3I_{|n_j|}(y)} g_j^0(x)\cot\left(\frac{y-x}{2}\right) dx\right|
 \left|\int_{K\setminus 3I_{|n_j|}(y)} g_j^0(z)\cot\left(\frac{y-z}{2}\right) dz\right| dy\\
 & \le C\lambda^2 n_j|J| |K|^2 \int_{(\Delta_j\cap \gamma J)\setminus\gamma K}
 \frac{1}{\sin^2\left(\frac{y-z_0}{2}\right)} dy.
\end{split}
\]
Following the already known steps, by Lemma \ref{gammaIgammaJ}, by the inequality $n_j\beta_j \le 40N\log^2(N+1)$ and by the fact that the maximal (with respect to $\F{}\setminus \F{\beta_j}$) dyadic intervals $\tilde J^{(k)}$ ($k$ is fixed) are disjoint we get

\[
\begin{split}
   & \sum_{J\in \F{}, |J|\le \beta_j} n_j|J| |K|^2 \int_{(\Delta_j\cap  \gamma J)\setminus\gamma K}
 \frac{1}{\sin^2\left(\frac{y-z_0}{2}\right)} dy \\
   & \le
\sum_{k=1/2-\gamma/2}^{\gamma/2-1/2}\sum_{\tilde J\in\F{}\setminus \F{\beta_j}, \tilde J^{(k)} \mbox{\ {\tiny maximal}}}\sum_{J\in\F{},  J^{(k)} \subset\tilde J^{(k)}}
   n_j|J| |K|^2 \int_{(\Delta_j\cap  J^{(k)})\setminus\gamma K}
 \frac{1}{\sin^2\left(\frac{y-z_0}{2}\right)} dy \\
   & \le
\sum_{k=1/2-\gamma/2}^{\gamma/2-1/2}\sum_{\tilde J\in\F{}\setminus \F{\beta_j}, \tilde J^{(k)} \mbox{\ {\tiny maximal}}}\sum_{J\in\F{},  J^{(k)} \subset\tilde J^{(k)}}
   n_j|J| |K|^2 \int_{(\Delta_j\cap  \tilde J^{(k)})\setminus\gamma K}
 \frac{1}{\sin^2\left(\frac{y-z_0}{2}\right)} dy \\
   & \le
C_{\gamma}\sum_{k=1/2-\gamma/2}^{\gamma/2-1/2}\sum_{\tilde J\in\F{}\setminus \F{\beta_j}, \tilde J^{(k)} \mbox{\ {\tiny maximal}}}
   n_j|\tilde J| |K|^2 \int_{(\Delta_j\cap  \tilde J^{(k)})\setminus\gamma K}
 \frac{1}{\sin^2\left(\frac{y-z_0}{2}\right)} dy \\
   & \le
C_{\gamma}\sum_{k=1/2-\gamma/2}^{\gamma/2-1/2}\sum_{\tilde J\in\F{}\setminus \F{\beta_j}, \tilde J^{(k)} \mbox{\ {\tiny maximal}}}
   N\log^2 (N+1) |K|^2 \int_{(\Delta_j\cap  \tilde J^{(k)})\setminus\gamma K}
 \frac{1}{\sin^2\left(\frac{y-z_0}{2}\right)} dy \\
 & \le
C_{\gamma}  N\log^2 (N+1) |K|^2 \int_{\Delta_j\setminus\gamma K}
 \frac{1}{\sin^2\left(\frac{y-z_0}{2}\right)} dy.
\end{split}
\]

Recall that
\begin{equation}\label{Delta_j}
\Delta_j = \left(\bigcup_{I\in\F{}, |I|> \frac{16}{n_j}}\gamma I \setminus \bigcup_{I\in\F{}, |I|>\beta_j}\gamma I\right)
= \gamma F_{16/n_j}\setminus\gamma F_{\beta_j}.
\end{equation}

We prove that for $k\ge \log^2 N$ ($N\ge 32$) $\Delta_{j+k}$ and  $\Delta_{j}$ are disjoint for $j\le N$.
This follows from the lacunarity of the  sequences $(n_j)$, that is, from $n_{j+1} \ge 2n_j$, $n_{j+k}\ge 2^{k}n_j > 4(j+k+1)\log^2(j+k+1)n_j$ because
$2^{k} >  4(j+k+1)\log^2(j+k+1)$ for every $k\ge \log^2N, j\le N$ and $N\ge 32$.
That is, in this case $\beta_{j+k} < \frac{16}{n_j}$ and this shows that $\Delta_{j+k}\cap\Delta_j=\emptyset$ for $k\ge \log^2 N$.
This implies $\Delta_{a\lceil\log^2 N\rceil +b}\cap \Delta_{(a-1)\lceil \log^2 N\rceil +b}=\emptyset, b<\lceil\log^2 N\rceil, a\lceil\log^2 N\rceil+b\le N$ and consequently

\[
\begin{split}
    & C_{\gamma}\sum_{j=1}^{N} N\log^2 (N+1) \lambda^2 \sum_{K\in \F{}}|K|^2 \int_{\Delta_j\setminus\gamma K}
 \frac{1}{\sin^2\left(\frac{y-z_0}{2}\right)} dy\\
     & \le C_{\gamma}N\log^2 (N+1) \lambda^2 \sum_{K\in \F{}}|K|^2 \log^2 (N+1) \int_{T\setminus \gamma K}\frac{1}{\sin^2\left(\frac{y-z_0}{2}\right)} dy \\
     & \le C_{\gamma}N\log^4 (N+1) \lambda^2 \sum_{K\in \F{}}|K|
      \le C_{\gamma}N\log^4 (N+1) \|f\|_1 \lambda.
\end{split}
\]
Summarizing our achievements we get
\[
\sum_{j=1}^{N} A_j^{2} = \sum_{j=1}^{N} A_j^{2,1}  + \sum_{j=1}^{N} A_j^{2.2}  \le C_{\gamma}N\log^4 (N+1) \|f\|_1 \lambda
\]
and similarly $\sum_{j=1}^{N} A_j^3 \le C_{\gamma}N\log^4 (N+1) \|f\|_1 \lambda$.

\textbf{Finally, investigate $A_j^1$.}
That is, give an estimation for the sum of integrals
\[
\begin{split}
    &
\sum_{J,K\in \F{}}\int_{(\gamma F\setminus\gamma F_{\beta_j})\cap \gamma J\cap \gamma K}
\int_{J\setminus 3I_{|n_j|}(y)} g_j^0(x)\cot\left(\frac{y-x}{2}\right) dx\\
& \times\int_{K\setminus 3I_{|n_j|}(y)} g_j^0(z)\cot\left(\frac{y-z}{2}\right) dz dy
 =:A_j^1.
\end{split}
\]
We split the set
\[
\gamma F\setminus\gamma F_{\beta_j}  = (\gamma F\setminus\gamma F_{16/n_j}) \cup (\gamma F_{16/n_j}\setminus\gamma F_{\beta_j})
 = (\gamma F\setminus\gamma F_{16/n_j}) \cup \Delta_j
\]
as above again and estimate the integrals on the set $\gamma F\setminus\gamma F_{16/n_j}$ (this will be $A_j^{1,1}$) then on $\Delta_j$ (and that will be $A_j^{1,2}$).

If $y\in \gamma F\setminus\gamma F_{16/n_j}$, then $y\notin \gamma L$ for any $L\in \F{16/n_j}$. Consequently, $y\in (\gamma F\setminus\gamma F_{16/n_j}) \cap \gamma J\cap \gamma K$ gives
$|J|, |K| \le \frac{16}{n_j}$ and
\[
\begin{split}
    & \sum_{j=1}^{N} A_j^{1,1} \\
    & \le  \sum_{j=1}^{N}  \sum_{J,K\in \F{}}
    \int_{(\gamma F\setminus\gamma F_{16/n_j})\cap \gamma J\cap \gamma K}
\left|\int_{J\setminus 3I_{|n_j|}(y)} g_j^0(x)\cot\left(\frac{y-x}{2}\right) dx\right|\\
& \times \left|\int_{K\setminus 3I_{|n_j|}(y)} g_j^0(z)\cot\left(\frac{y-z}{2}\right) dz \right| dy \\
 &  \le C\sum_{j=1}^{N}  \sum_{J,K\in \F{}}
    \int_{(\gamma F\setminus\gamma F_{16/n_j})\cap \gamma J\cap \gamma K}
   n_j|J| \lambda   n_j|K| \lambda dy \\
     & \le C\lambda^2\sum_{j=1}^{N}\sum_{J,K\in \F{}} |(\gamma F\setminus\gamma F_{16/n_j})\cap \gamma J\cap \gamma K|\\
     & \le C_{\gamma}\lambda^2 N |F|\\
     &  \le C_{\gamma}N \|f\|_1\lambda
\end{split}
\]
as it comes from the method already used several times (see e.g. Lemma \ref{gammaIgammaJ}).

Finally, investigate
\[
\begin{split}
    & \sum_{j=1}^{N} A_j^{1,2} \\
    & \le  \sum_{j=1}^{N}  \sum_{J,K\in \F{}}
    \int_{(\gamma F_{16/n_j}\setminus\gamma F_{\beta_j})\cap \gamma J\cap \gamma K}
\left|\int_{J\setminus 3I_{|n_j|}(y)} g_j^0(x)\cot\left(\frac{y-x}{2}\right) dx\right|\\
& \times \left|\int_{K\setminus 3I_{|n_j|}(y)} g_j^0(z)\cot\left(\frac{y-z}{2}\right) dz \right| dy.
\end{split}
\]

Since $y\in (\gamma F_{16/n_j}\setminus\gamma F_{\beta_j})\cap \gamma J = \Delta_j\cap \gamma J$, one has $y\notin \gamma F_{\beta_j}$ and consequently
$J\notin \F{\beta_j}$, that is, $|J|\le \beta_j$. By this fact we have

\[
\left|\int_{J\setminus 3I_{|n_j|}(y)} g_j^0(x)\cot\left(\frac{y-x}{2}\right) dx\right| \le Cn_j|J|\lambda.
\]
Later on, we also will use that $ n_j|J|\le n_j\beta_j = 20(j+1)\log^2 (j+1) \le 40N\log^2 (N+1)$.
Thus,  the already known method gives an estimation for
\[
\sum_{j=1}^{N} A_j^{1,2} \le C\lambda\sum_{j=1}^{N}\sum_{J,K\in \F{}, |J|\le \beta_j}n_j|J|\int_{\Delta_j\cap \gamma J\cap \gamma K}
    \left|\int_{K\setminus 3I_{|n_j|}(y)} g_j^0(z)\cot\left(\frac{y-z}{2}\right) dz \right| dy.
\]

That is,
\[
\begin{split}
    & \sum_{j=1}^{N} A_j^{1,2}
     \le  C\lambda \sum_{j=1}^{N}\sum_{K\in \F{}}
\sum_{k=1/2-\gamma/2}^{\gamma/2-1/2}\sum_{\tilde J\in\F{}, |\tilde J|\le \beta_j, \tilde J^{(k)} \mbox{\ {\tiny maximal}}}\sum_{J\in\F{},  J^{(k)} \subset\tilde J^{(k)}}
     n_j|J| \\
    & \times
    \int_{\Delta_j\cap J^{(k)}\cap \gamma K}
    \left|\int_{K\setminus 3I_{|n_j|}(y)} g_j^0(z)\cot\left(\frac{y-z}{2}\right) dz \right| dy \\
& \le  C_{\gamma}\lambda \sum_{j=1}^{N}\sum_{K\in \F{}}\sum_{k=1/2-\gamma/2}^{\gamma/2-1/2}
\sum_{\tilde J\in\F{}, |\tilde J|\le \beta_j, \tilde J^{(k)} \mbox{\ {\tiny maximal}}}n_j|\tilde J|
    \\
& \times    \int_{\Delta_j\cap \tilde J^{(k)}\cap \gamma K}
    \left|\int_{K\setminus 3I_{|n_j|}(y)} g_j^0(z)\cot\left(\frac{y-z}{2}\right) dz \right| dy \\
&  \le  C_{\gamma}\lambda N\log^2 (N+1) \sum_{j=1}^{N} \sum_{K\in \F{}}\sum_{k=1/2-\gamma/2}^{\gamma/2-1/2}
\sum_{\tilde J\in\F{}, |\tilde J|\le \beta_j, \tilde J^{(k)} \mbox{\ {\tiny maximal}}}\\
&  \times \int_{\Delta_j\cap \tilde J^{(k)}\cap \gamma K}
    \left|\int_{K\setminus 3I_{|n_j|}(y)} g_j^0(z)\cot\left(\frac{y-z}{2}\right) dz \right| dy \\
&  \le  C_{\gamma}\lambda N\log^2 (N+1) \sum_{K\in \F{}}\sum_{j=1}^{N}
        \int_{\Delta_j\cap  \gamma K}
    \left|\int_{K\setminus 3I_{|n_j|}(y)} g_j^0(z)\cot\left(\frac{y-z}{2}\right) dz \right| dy \\
    \end{split}
\]
because by Lemma \ref{gammaIgammaJ} we have $\sum_{J\in\F{},  J^{(k)} \subset\tilde J^{(k)}}|J|\le C_{\gamma} |\tilde J^{(k)}| \le C_{\gamma}\beta_j$,
and because the $\F{}\setminus \F{\beta_j}$ maximal dyadic intervals $\tilde J^{(k)}$ are disjoint.

 Recall that the sets $\Delta_j$ are not necessarily disjoint, but ``nearly disjoint''. That is, $\Delta_{j+k}\cap \Delta_j=\emptyset$ for $k\ge \log^2 N$
 ($N\ge j, 32$).
 Besides, for each fixed $K\in \F{}$ there are at most $(\gamma +1)\log^2 N$ 
  $\Delta_j$-s such that $\gamma K\cap \Delta_j\not=\emptyset$. This comes as follows.
 We have two possible situations. $|K|>16/n_j$ and $|K|\le 16/n_j$.
 Let $j$ be the smallest index for which  $|K|>16/n_j$ and $\gamma K\cap \Delta_j\not= \emptyset$. In this case $|K|>16/n_j>\beta_{j+k}$
 for $k\ge \log^2 N$ and consequently $\gamma K\cap \Delta_{j+k} = \emptyset$.
If $|K|\le 16/n_j$  (and $\gamma K\cap \Delta_j\not= \emptyset$), then one of the $K^{(i)}$'s, say $K^{(i_j)}$ must be part of $\Delta_j$, and so those $\Delta_j$'s for which $i_j$ is the same, will intersect,
so the claim follows from the earlier proven fact about the disjointness of the different  $\Delta_j$'s.

Take such a $\Delta_j$ and let $y\in \gamma K\cap \Delta_j$ and $z\in K\setminus 3I_{|n_j|}(y)$. This latter gives
 $|y-z|> 2\pi/2^{|n_j|} > 1/n_j$. Besides, $z\in K, y\in \gamma K$ gives $|y-z|\le \gamma |K|\le \gamma \beta_j$ because $\gamma K\cap \Delta_j\not=\emptyset$ also means $|K|\le \beta_j$.

That is, $ 1/n_j < |y-z|\le \gamma\beta_j$ and consequently
\[
\begin{split}
    &
\int_{\Delta_j\cap \gamma K}
    \left|\int_{K\setminus 3I_{|n_j|}(y)} g_j^0(z)\cot\left(\frac{y-z}{2}\right) dz \right| dy\\
& \le \int_{K} |g_j^0(z)|\int_{\set{y: \frac{1}{n_j} < |y-z|\le \gamma\beta_j}} \left|\cot\left(\frac{y-z}{2}\right) \right| dy dz\\
& \le C_{\gamma}\int_{K} |g_j^0(z)|\log (\beta_jn_j) dz \\
& \le C_{\gamma}\lambda |K|\log (N+1).
\end{split}
\]
Moreover,  by what is written above we get
\[
\sum_{j=1}^{N}\int_{\Delta_j\cap \gamma K}
    \left|\int_{K\setminus 3I_{|n_j|}(y)} g_j^0(z)\cot\left(\frac{y-z}{2}\right) dz \right| dy \le C_{\gamma}|K|\lambda \log^3 (N+1).
\]
This inequality finally gives
\[
\sum_{j=1}^{N}A_j^{1,2}  \le C_{\gamma}N\log^5 (N+1)\|f\|_1\lambda.
\]
That is, the proof of Lemma \ref{gammaF-gammaFjH} is complete.
\end{proof}
We go further in section $4$ in the investigation of integrals on the set $\gamma F\setminus\gamma F_{\beta_j}$. The next lemma to be proved is:

\begin{lem}\label{gammaF-gammaFjS}
Let $\gamma >5$ be an odd integer,  $N\in\N$. Let $f\in L^1(T), \lambda >\|f\|_1/(2\pi)$, $n_j\le l_j\le 2n_j$ be natural numbers  $n_j\beta_j = 20(j+1)\log^2 (j+1)$ for $j=1,\dots, N$.
Then the equality
\[
\sum_{j=1}^{N}\int_{\gamma F\setminus\gamma F_{\beta_j}}|S_{l_j} f(y)|^2 dy \le C_{\gamma} N\log^5 (N+1)\|f\|_1\lambda
\]
holds. The constant $C_{\gamma}$ can depend only on $\gamma$ (and it is uniform in $f, (l_j), (n_j), N, \lambda$).
\end{lem}
\begin{proof}
Similarly, as in the proof of Lemma \ref{gammaFjS}, but not the very same way (compare with (\ref{modified_partial}) and see the domain of the integral) we define
\[
\tilde S_{l_j} f(y) := \frac{1}{\pi}\int_{T\setminus 3I_{|n_j|}(y)} f(x)D_{l_j}(y-x) dx.
\]
From the definition of the Hilbert transform (\ref{hilbert_def})
in the same way as  (\ref{^2estimation_on_modified_partial}) is given we have
\[
|\tilde S_{l_j} f(y) |^2 \le \|f\|_1^2 + \left|H_{n_j}(f(\cdot)e^{-\imath(l_j+1)\cdot})(y)\right|^2 + \left|H_{n_j}(f(\cdot)e^{\imath l_j\cdot})(y)\right|^2.
\]
For $g_j(x) = f(x)e^{-\imath(l_j+1)x}$ and $h_j(x) = f(x)e^{\imath l_jx}$ we can apply  Lemma \ref{cz} and then we get that the set $\F{}$ is the same for $g_j, h_j$ and $f$
since their absolute values coincide. Then we can apply Lemma \ref{gammaF-gammaFjH} for them and this implies
\[
\begin{split}
   &
\sum_{j=1}^{N}\int_{\gamma F\setminus\gamma F_{\beta_j}}|\tilde S_{l_j} f(y)|^2 dy \\
& \sum_{j=1}^{N}\left(\|f\|_1^2 + \int_{\gamma F\setminus\gamma F_{\beta_j}}|H_{n_j} g_j(y)|^2 dy
+ \int_{\gamma F\setminus\gamma F_{\beta_j}}|H_{n_j} h_j(y)|^2 dy \right)\\
& \le N\|f\|_1^2 +  C_{\gamma} N\log^5 (N+1)\|f\|_1\lambda \le C_{\gamma} N\log^5 (N+1)\|f\|_1\lambda.
\end{split}
\]
Now, we have to check the difference of $S_{l_j}f$ and $\tilde S_{l_j}f$. Setting
\[
E_{l_j}f(y) := l_j\int_{3I_{|n_j|}(y)}f(x) dx
\]
(a bit different then it was in (\ref{Eldef})) we have (similarly as in (\ref{SltildeSldifference}))

\[
\left|S_{l_j}f(y)-\tilde S_{l_j}f(y)\right|
\le
\frac{1}{\pi}\left|\int_{3I_{|n_j|}(y)}f(x)D_{l_j}(y-x) dx\right| \le \frac{l_j+1/2}{\pi}\int_{3I_{|n_j|}(y)}|f(x)| dx \le E_{l_j}|f|(y).
\]
In the sequel we prove
\[
\sum_{j=1}^{N}\int_{\gamma F\setminus\gamma F_{\beta_j}}|E_{l_j} f(y)|^2 dy \le C_{\gamma} N\log^4 (N+1)\|f\|_1\lambda
\]
for every $f\in L^1(T)$. This inequality applied to the  function $|f|\in L^1(T)$ would complete the proof of Lemma \ref{gammaF-gammaFjS}.
(Set $\F{}$ (with fixed $\lambda$) is the same for $f$ and $|f|$.)

See again Lemma \ref{cz} for $f=\sum_{i=0}^{\infty}f_i = f_0+f^0$. Similarly, as in the proof of Lemma \ref{gammaFjS} we have
$|E_{l_j}f_0(y)|^2 \le 24\pi\|f_0\|_{\infty} |E_{l_j}f_0(y)|$ and consequently,
\[
\int_{\gamma F\setminus\gamma F_{\beta_j}}|E_{l_j} f_0(y)|^2 dy \le C\lambda\|E_{l_j} f_0\|_1 \le C\lambda\|f\|_1.
\]

Next, investigate $f^0$.

As earlier in this paper
\[
\begin{split}
     &
\sum_{j=1}^{N}\int_{\gamma F\setminus\gamma F_{\beta_j}}|E_{l_j} f^0(y)|^2 dy \\
& \le
C\sum_{j=1}^{N}\sum_{J,K\in\F{}} \int_{\gamma F\setminus\gamma F_{\beta_j}}l_j^2
\left|\int_{J\cap 3I_{|n_j|}(y)}f^0(x) dx\right|
\left|\int_{K\cap 3I_{|n_j|}(y)}f^0(z) dz\right| dy =: \sum_{j=1}^{N}A_j.
\end{split}
\]

The sum over pairs $(J,K)\in \F{}\times\F{}$ and the integral $\int_{\gamma F\setminus\gamma F_{\beta_j}}$ will be divided into the following parts:
\[
\begin{split}
    & \sum_{J,K\in \F{}}\int_{\gamma F\setminus\gamma F_{\beta_j}} = \sum_{J,K\in \F{}}\int_{(\gamma F\setminus\gamma F_{\beta_j})\cap \gamma J\cap\gamma K} \\
&    + \sum_{J,K\in \F{}}\int_{((\gamma F\setminus\gamma F_{\beta_j})\cap \gamma J)\setminus\gamma K} \\
 &   + \sum_{J,K\in \F{}}\int_{((\gamma F\setminus\gamma F_{\beta_j})\cap \gamma K)\setminus\gamma J}\\
  &  + \sum_{J,K\in \F{}}\int_{(\gamma F\setminus\gamma F_{\beta_j})\setminus (\gamma J\cup\gamma K)} =:\sum_{i=1}^{4}A_j^i.
    \end{split}
    \]
First, check $A_j^3$.

In this case $y\in (\gamma F\setminus\gamma F_{\beta_j})\setminus \gamma J$. That is, $y\notin \gamma J$. If
$|I_{|n_j|}(y)| =2\pi/2^{|n_j|}  \le |J|$, then $3I_{|n_j|}(y)\cap J =\emptyset$ and consequently
$\int_{J\cap 3I_{|n_j|}(y)}f^0(x) dx=0$.
If $|I_{|n_j|}(y)| > |J|$, then $I_{|n_j|}(y) \subset J$ is not possible and consequently either $J\cap I_{|n_j|}(y) = J$ or $J\cap I_{|n_j|}(y) = \emptyset$.  In both cases
$\int_{J\cap I_{|n_j|}(y)}f^0(x) dx=0$. The same can be said about the intervals $ I^+_{|n_j|}(y)$ and $ I^-_{|n_j|}(y)$. That is,
$\int_{J\cap 3I_{|n_j|}(y)}f^0(x) dx=0$ in every case when $y\notin \gamma J$.
This gives that $A_j^3=0$ and also that  $A_j^4=0$ and similarly (applying this procedure for $K$ instead of $J$) that  $A_j^2=0$.
That is, it is remained to investigate
\[
\begin{split}
     &
\sum_{j=1}^{N}A_j = \sum_{j=1}^{N}A_j^1 \\
& = \sum_{j=1}^{N}\sum_{J,K\in \F{}}
\int_{(\gamma F\setminus\gamma F_{\beta_j})\cap \gamma J\cap\gamma K}l_j^2
\left|\int_{J\cap 3I_{|n_j|}(y)}f^0(x) dx\right|
\left|\int_{K\cap 3I_{|n_j|}(y)}f^0(z) dz\right| dy.
\end{split}
\]
Recall that
\[
\gamma F\setminus\gamma F_{\beta_j} = \left(\bigcup_{I\in\F{}}\gamma I \setminus \bigcup_{I\in\F{}, |I|>\frac{16}{n_j}}\gamma I\right)
\cup \left(\bigcup_{I\in\F{}, |I|> \frac{16}{n_j}}\gamma I \setminus \bigcup_{I\in\F{}, |I|>\beta_j}\gamma I\right)
 = \left(\gamma F\setminus\gamma F_{16/n_j}\right) \cup \Delta_j.
 \]

 Either $|J|\le 16/n_j$ or  $(\gamma F\setminus\gamma F_{16/n_j})\cap \gamma J=\emptyset$. In other words, if we integrate on
  $(\gamma F\setminus\gamma F_{16/n_j})\cap \gamma J\cap\gamma K$, then we can suppose that $|J|, |K|\le 16/n_j$.
  This by inequalities
  \[
  |J|^{-1}\int_{J}|f^0|, \quad |K|^{-1}\int_{K}|f^0| \le 4\lambda
  \]
  (see Lemma \ref{cz}) and by Lemma \ref{gammaIgammaJ} implies that

\[
\begin{split}
     &\sum_{j=1}^{N}\sum_{J,K\in \F{}}\int_{(\gamma F\setminus\gamma F_{16/n_j})\cap \gamma J\cap\gamma K}l_j^2
\left|\int_{J\cap 3I_{|n_j|}(y)}f^0(x) dx\right|
\left|\int_{K\cap 3I_{|n_j|}(y)}f^0(z) dz\right| dy\\
&\le \sum_{j=1}^{N}\sum_{J,K\in \F{}}\int_{(\gamma F\setminus\gamma F_{16/n_j})\cap \gamma J\cap\gamma K} 2^{14}\lambda^2 dy\\
&\le 2^{14}\lambda^2\sum_{j=1}^{N} \sum_{J,K\in \F{}}|(\gamma F\setminus\gamma F_{16/n_j})\cap\gamma J\cap\gamma K|
\le 2^{14}\lambda^2\sum_{j=1}^{N} \sum_{J,K\in \F{}}|\gamma J\cap \gamma K|\\
& \le C_{\gamma}\lambda^2N |F|
\le C_{\gamma}N\|f\|_1\lambda.
\end{split}
\]
Now, we have to check the integrals $\int_{\Delta_j} dy$.  In this case, $y\in \Delta_j = \gamma F_{16/n_j}\setminus\gamma F_{\beta_j}$.
Thus, by the next lemma, that is by Lemma \ref{deltaint} below the proof of Lemma \ref{gammaF-gammaFjS} is complete.
\end{proof}

In order to complete the proof of Lemma \ref{gammaF-gammaFjS} we need one additional step which is written in the next lemma. It also uses the notation of
Lemma \ref{gammaF-gammaFjS} and other parts of this paper but from the point of view of
readability, we give this additional step in a separate lemma. 

\begin{lem}\label{deltaint}
\[
\sum_{j=1}^{N}\
\sum_{J,K\in \F{}}\int_{\Delta_j\cap \gamma J\cap\gamma K}n_j^2
\left|\int_{J\cap 3I_{|n_j|}(y)}f^0(x) dx\right|
\left|\int_{K\cap 3I_{|n_j|}(y)}f^0(z) dz\right| dy
 \le C_{\gamma}N\lambda\log^4 (N+1)\|f\|_1.
 \]
 The constant $C_{\gamma}$ can depend only on $\gamma$ (and it is uniform in $f,  (n_j), N$ and  $\lambda$).
\end{lem}
\begin{proof}
With the same argument as in the very beginning of Lemma \ref{gammaF-gammaFjH}, we can suppose that $N\ge 32$ again. In this proof - if it does not cause misunderstanding- $I_{|n_j|}(y)$ is simply denoted by $I$.
Set ($\frac{2\pi}{2^{|n_j|}} < 16/n_j$)
\[
\tilde \Delta_j := \gamma F_{\frac{2\pi}{2^{|n_j|}}} \setminus \gamma F_{\beta_j} \supset  \gamma F_{\frac{16}{n_j}} \setminus \gamma F_{\beta_j} = \Delta_j,
\]
and  (recall the definition of  $\Delta_j$ at (\ref{Delta_j}))
\[
\Delta_j^{\prime} := \{y\in \tilde\Delta_j : n_j\int_{3I_{|n_j|}(y)}|f(z)| dz > 50\lambda\}, \quad
\Delta_j^{\prime\prime} := \tilde\Delta_j \setminus \Delta_j^{\prime}.
\]
Remark that the set $\Delta_j^{\prime}$  is the union of dyadic intervals of length $2\pi/2^{|n_j|}$.
Then by the relations $f^0 = f - f_0, \|f_0\|_{\infty}\le 2\lambda$ we have
\[
\begin{split}
    & \sum_{j=1}^{N}
\sum_{J,K\in \F{}}\int_{\Delta_j^{\prime\prime}\cap \gamma J\cap\gamma K}n_j^2
\left|\int_{J\cap 3I}f^0(x) dx\right|
\left|\int_{K\cap 3I}f^0(z) dz\right| dy \\
     & \le C\lambda^2 \sum_{j=1}^{N}\sum_{J,K\in \F{}}|\Delta_j^{\prime\prime}\cap \gamma J\cap\gamma K| \le   C\lambda^2 \log^2 N
     \sum_{J,K\in \F{}}|\gamma J\cap\gamma K|
      \le C_{\gamma}\lambda  \log^2 N \|f\|_1,
\end{split}
\]
where the last but one inequality comes from the fact that $\Delta_j^{\prime\prime}\subset \tilde\Delta_j$ and $\Delta_k^{\prime\prime}\subset \tilde\Delta_k$ are disjoint
for $|j-k|> \log^2 N$ and $N\ge j,k,32$ (similarly as in the case of $\Delta_j, \Delta_k$ for $|j-k|> \log^2 N$) and the last  inequality comes from Lemma \ref{gammaIgammaJ}.

Now, let $|a|, |b|<\gamma/2$ be integers and
investigate for some $J,K\in \F{}$:
\[
A_{a,b} := \int_{\Delta_j^{\prime}\cap J^{(a)}\cap K^{(b)}}n_j^2 \int_{J\cap 3I}|f^0(x)| dx \int_{K\cap 3I}|f^0(z)| dz dy.
\]
If $|J|> \beta_j$, then by definition $\tilde \Delta_j\cap \gamma J=\emptyset$ and then $\Delta_j^{\prime}\cap J^{(a)}\cap K^{(b)}=\emptyset$. That is,
$|J|, |K| \le  \beta_j$ can be supposed.

If $|J| \le 16/n_j$, then   $n_j\int_{J\cap 3I}|f^0(x)| dx \le n_j |J|4\lambda \le 64\lambda$,
$n_j\int_{K\cap 3I}|f^0(z)| dz \le 4\lambda n_j\beta_j \le 160\lambda N\log^2 (N+1)$.
Thus, we have
$A_{a,b} \le C\lambda^2 N\log^2 (N+1) |\Delta_j^{\prime}\cap J^{(a)}\cap K^{(b)}|$ and
\[
\sum_{j=1}^{N}
\sum_{K\in \F{}}\sum_{J\in \F{}, |J|\le 16/n_j}|\Delta_j^{\prime}\cap J^{(a)}\cap K^{(b)}| \le
\sum_{J, K\in \F{}}\sum_{j=1}^{N}|\Delta_j^{\prime}\cap J^{(a)}\cap K^{(b)}| \le C_{\gamma}\log^2 N |F|.
\]
(Recall that for $|j-k|>\log^2 N$ and $j,k\le N$ we have $\Delta^{\prime}_j\cap \Delta^{\prime}_{k}=\emptyset$.)
The same can be said if $|K| \le 16/n_j$. That is, from now on we can suppose that $\beta_j \ge |J|, |K| > 16/n_j$.

\medskip
Let $\mathcal{A}_l$ be the $\sigma$-algebra generated by the dyadic intervals with measure  $2\pi/2^l$ ($l\in\mathbb{N}$).
For a fixed $j$, $e_j:=2\pi/2^{|n_j|}$,  $\Delta_j^{\prime}-e_j = \{y-e_j : y\in \Delta_j^{\prime}\}$ (and similarly for $\Delta_j^{\prime}+e_j$).  We prove
for $|j-k|> \log^2 N$, $\epsilon,\delta =-1,0,1$ that the sets $\Delta_j^{\prime} + \epsilon e_j$ and $\Delta_k^{\prime}+\delta e_k$ are disjoint,

Suppose that $k>j+\log^2 N$. If $z\in \Delta_j^{\prime} + \epsilon e_j$, then the distance of $z$ and $\Delta_j^{\prime}$ is not more than $2\pi/2^{|n_j|}$. Consequently at least one of $I_{|n_j|}(z), I^-_{|n_j|}(z), I^+_{|n_j|}(z)$ is a subset of the $\Delta_j^{\prime}\in \mathcal{A}_{|n_j|}$ measurable set.
Say, $I^-_{|n_j|}(z)$ is this set. Then for $y = z-e_j$ we have $y\in \Delta_j^{\prime}$ and thus $n_j\int_{3I_{|n_j|}(y)}|f| > 50\lambda$ and consequently there is at least one
set among $I_{|n_j|}(y), I^-_{|n_j|}(y), I^+_{|n_j|}(y)$ say $I^-_{|n_j|}(y)$ for which $2^{|n_j|}/(2\pi)\int_{I^-_{|n_j|}(y)}|f| > \lambda$. This by Lemma \ref{cz} and by the definition
of $\F{}$ gives the existence of a dyadic interval $L\in \F{}$ such that $I^-_{|n_j|}(y) \subset L$. Then $|L| \ge 2\pi/2^{|n_j|} > \beta_{k}$ for $k>j+\log^2 N$. Then by definition we have $\gamma L\cap \tilde\Delta_k=\emptyset$. Then  $\gamma L\cap \Delta_k^{\prime}=\emptyset$. This implies $(\gamma-2) L\cap \Delta_k^{\prime}+\delta e_k=\emptyset$
for $\delta=-1,0,1$. But $y\in 3L$ and then $z\in 5L$ and $\gamma \ge 7$ gives that $z\notin \Delta_k^{\prime}+\delta e_k$.
That is, we proved that
\[
 \Delta_j^{\prime} + \epsilon e_j\cap \Delta_k^{\prime}+\delta e_k = \emptyset \quad (\epsilon,\delta =-1,0,1, |j-k|> \log^2 N).
\]
Go back to check $A_{a,b}$. Then $y\in \Delta_j^{\prime}\cap J^{(a)}\cap K^{(b)}$. If $|a|>1$, then $y\notin 3J$ and then $J\cap 3I_{|n_j|}(y)=\emptyset$ and thus
$A_{a,b}=0$. The same can be said if $|b|>1$. That is, $|a|, |b|\le 1$ can be supposed.

\textbf{Case $1$.} From now on we suppose that $J\not= K$ (which also means $J\cap K = \emptyset$). Recall that in this proof - if it does not cause misunderstanding, $I_{|n_j|}(y)$ is simply denoted by $I$.
One of $I\cup I^+$ and $I\cup I^-$ should be   $I_{|n_j|-1}(y)$. Say, this is $I\cup I^+$ and this should be a subset of $J$ or a subset of $K$ 
(otherwise at least
one of $J\cap 3I$  and $K\cap 3I$ is the empty set and the corresponding integral on $J\cap 3I$ or on $K\cap 3I$ is zero). Say, $I\cup I^+ = I_{|n_j|-1}(y)\subset K$. Then $I^-\subset J$ should hold. This means $y\in K = K^{(0)}, y\in J+e_j$ and $y\notin J$
($J\cap K = \emptyset$). Thus, $y\in J^+ = J^{(1)}$. This case can happen only if $a=1, b=0$. Consequently, we have to check the cases only for
$(a,b)=\{(1,0), (-1,0), (0,1), (0,-1)\}$. Check the case $(a,b) = (1,0)$. Then by the relations $n_j\int_{K\cap 3I}|f^0(z)| dz \le 4n_j|K| \lambda \le 4n_j\beta_j\lambda \le CN\log^2 (N+1) \lambda$ and by the disjointness of the elements of $\F{}$ (recall that now $16/n_j < |J|, |K| \le \beta_j$)
\[
\begin{split}
&
B:= \sum_{K\in \F{16/n_j}\setminus \F{\beta_j}}A_{1,0} \le \sum_{K\in \F{16/n_j}\setminus \F{\beta_j}}
\int_{\Delta_j^{\prime}\cap K}n_j^2 \int_{J\cap 3I}|f^0(x)| dx \int_{K\cap 3I}|f^0(z)| dz dy\\
& \le C\lambda N\log^2 (N+1)  \sum_{K\in \F{}}\int_{\Delta_j^{\prime}\cap K}n_j \int_{J\cap 3I}|f^0(x)| dx  dy\\
& \le C\lambda N\log^2 (N+1) \int_{\Delta_j^{\prime}}n_j \int_{J\cap 3I}|f^0(x)| dx  dy.
\end{split}
\]
Since $J, \Delta_j^{\prime}$ are  measurable with respect to $\mathcal{A}_{|n_j|}$ ($J\in \F{}, |J|>16/n_j > 2\pi/2^{|n_j|}$) consequently  we have for every $g\in L^1$:
\begin{equation}\label{intDeltaJ3I}
\begin{split}
&
\int_{\Delta_j^{\prime}}\frac{2^{|n_j|}}{2\pi}\int_{J\cap I_{|n_j|}(y)}|g(x)| dx  dy = \int_{\Delta_j^{\prime}\cap J} |g(y)| dy,\\
& \int_{\Delta_j^{\prime}}\frac{2^{|n_j|}}{2\pi}\int_{J\cap I^-_{|n_j|}(y)}|g(x)| dx  dy = \int_{(\Delta_j^{\prime}-e_j)\cap J} |g(y)| dy,\\
& \int_{\Delta_j^{\prime}}\frac{2^{|n_j|}}{2\pi}\int_{J\cap I^+_{|n_j|}(y)}|g(x)| dx  dy = \int_{(\Delta_j^{\prime}+e_j)\cap J} |g(y)| dy.
\end{split}
\end{equation}
Thus, taking into account the estimation for  $B$ above:
\[
\begin{split}
& \sum_{j=1}^{N}\sum_{J, K\in \F{16/n_j}\setminus \F{\beta_j}}A_{1,0} \\
& \le
C\lambda N\log^2 (N+1) \sum_{j=1}^{N}\sum_{J\in \F{}}\left(\int_{\Delta_j^{\prime}\cap J} |f^0(y)| dy + \int_{(\Delta_j^{\prime}-e_j)\cap J} |f^0(y)| dy + \int_{(\Delta_j^{\prime}+e_j)\cap J} |f^0(y)| dy\right)\\
& \le
C\lambda N\log^2 (N+1) \sum_{j=1}^{N}\left(\int_{\Delta_j^{\prime}} |f^0(y)| dy + \int_{(\Delta_j^{\prime}-e_j)} |f^0(y)| dy + \int_{(\Delta_j^{\prime}+e_j)} |f^0(y)| dy\right) \\
&  \le CN\log^4 (N+1) \|f\|_1\lambda,
\end{split}
\]
since as we proved above  $(\Delta_j^{\prime}+\epsilon e_j)\cap (\Delta_k^{\prime}+\delta e_k) =\emptyset$ for any $|j-k|>\log^2 N$.
That is, the investigation of the case $1$ is done.

\textbf{Case $2$.} From now on we suppose that $J= K$. Then in $A_{a,b}$ it can be supposed that $a=b$.

\[
A_{a,a} \le C\lambda N\log^2 (N+1) \int_{\Delta_j^{\prime}\cap J^{(a)}}n_j \int_{J\cap 3I}|f^0(x)| dx  dy.
\]
Recall that $2\pi/2^{|n_j|} = |I|= |I_{|n_j|}(y)| < 16/n_j < |J|  \le \beta_j$ can be (and it is already) supposed. If $|a|>1$, then $y\in J^{(a)}$ gives $J\cap 3I =\emptyset$ and then $A_{a,a}=0$.
That is, $a=-1,0,1$. In the very same way as above in the case of $A_{a,b}$ by (\ref{intDeltaJ3I}) we get again in the case $J=K$
\[
A_{a,a} \le C\lambda N\log^2 (N+1) \left(\int_{\Delta_j^{\prime}\cap J} |f^0(y)| dy + \int_{(\Delta_j^{\prime}-e_j)\cap J} |f^0(y)| dy +\int_{(\Delta_j^{\prime}+e_j)\cap J} |f^0(y)| dy\right).
\]
Thus, summing up with respect to $J\in\F{}$ and $j$ we get the same bound as in Case $1$. The proof of Lemma \ref{deltaint} is complete.
\end{proof}

 A straightforward consequence of Lemma \ref{gammaF-gammaFjS} is

\begin{cor}\label{gammaF-gammaFjS-V}
Let $\gamma >5$ be an odd integer,  $N\in\N$. Let $f\in L^1(T)$, $\lambda >\|f\|_1/(2\pi)$, $(n_j)$ be a lacunary (with parameter not less than $2$) sequence of natural numbers and $n_j\beta_j = 20(j+1)\log^2 (j+1)$ for $j=1,\dots, N$.
Then the inequality
\[
\sum_{j=1}^{N}\int_{\gamma F\setminus\gamma F_{\beta_j}}|S_{n_j}f(y)-V_{n_j}f(y)|^2 dy \le C_{\gamma} N\log^5 (N+1)\|f\|_1\lambda
\]
holds.  The constant $C_{\gamma}$ can depend only on $\gamma$ (and it is uniform in $f,  (n_j), N$ and  $\lambda$).
\end{cor}

\begin{proof}
Use the inequality $|V_{n_j}f|^2 = \left|\frac{1}{n_j}\sum_{l_j=n_j}^{2n_j-1}S_{l_j}f\right|^2 \le \frac{1}{n_j}\sum_{l_j=n_j}^{2n_j-1}|S_{l_j}f|^2$.
By Lemma \ref{gammaF-gammaFjS} we have
\[
\begin{split}
    &  \sum_{j=1}^{N}\int_{\gamma F\setminus\gamma F_{\beta_j}}|V_{n_j}f(y)|^2 dy\\
     &  \le \sum_{j=1}^{N}\frac{1}{n_j}\sum_{l_j=n_j}^{2n_j-1}\int_{\gamma F\setminus\gamma F_{\beta_j}}|S_{l_j}f(y)|^2 dy\\
     & = \sum_{l_N=n_N}^{2n_N-1}\dots \sum_{l_1=n_1}^{2n_1-1}\frac{1}{n_N\cdots n_1}\sum_{j=1}^{N}\int_{\gamma F\setminus\gamma F_{\beta_j}}|S_{l_j}f(y)|^2 dy\\
     & \le \frac{1}{n_N\cdots n_1}\sum_{l_N=n_N}^{2n_N-1}\dots \sum_{l_1=n_1}^{2n_1-1} C_{\gamma}N\log^5 (N+1) \|f\|_1\lambda \le  C_{\gamma}N\log^5 (N+1) \|f\|_1\lambda.
\end{split}
\]
 Besides, apply Lemma \ref{gammaF-gammaFjS} to $S_{n_j}f$, that is, to the numbers $l_j=n_j$ and the proof of Corollary \ref{gammaF-gammaFjS-V} is complete.
\end{proof}

Summarize our achievements for the time being:
\begin{cor}\label{gammaFS-V}
Let $\beta>\gamma >5$ be  odd integers,  $N\in\N$. Let $f\in L^1(T), \lambda >\|f\|_1/(2\pi)$, $(n_j)$ be a lacunary (with parameter not less than $2$) sequence of natural numbers and $n_j\beta_j = 20(j+1)\log^2 (j+1)$, $n_j\le 50 m_j$ for $j=1,\dots, N$.
Then the inequality
\[
\sum_{j=1}^{N}\int_{\gamma F}|S_{n_j}f(y)-V_{n_j}f(y)|^2  \left|\sigma_{m_j}1_{\overline{\beta F_{\beta_j}}}(y)\right|^2 dy \le C_{\beta, \gamma} N\log^5 (N+1)\|f\|_1\lambda
\]
holds.  The constant $C_{\beta, \gamma}$ can depend only on $\gamma$ and $\beta$ (and it is uniform in $f,  (n_j), N$ and  $\lambda$).
\end{cor}

\begin{proof}
  For the sum of integrals $\int_{\gamma F_{\beta_j}} dy$ by Corollary \ref{gammaFjS-V} we have the estimation $C_{\beta,\gamma}N\|f\|_1\lambda$ and for
  the sum of integrals $\int_{\gamma F\setminus\gamma F_{\beta_j}} dy$ by the inequality $\|\sigma_{m_j}1_{\overline{\beta F_{\beta_j}}}\|_{\infty}\le 1$ and by Corollary \ref{gammaF-gammaFjS-V}  
  we have the estimation $C_{\gamma}N\log^5 (N+1)\|f\|_1\lambda$. This completes the proof of Corollary \ref{gammaFS-V}.
\end{proof}

The final ``integral section'' is
\section{The integral on $T\setminus\gamma F$}

Recall the definition of the Hilbert transform $H_n$ in (\ref{hilbert_def}).
The first lemma in this section  is:
\begin{lem}\label{T-gammaFH}
Let $\gamma >5$ be an odd integer, $n\in\N$  and $f\in L^1(T), \lambda >\|f\|_1/(2\pi)$. Then the inequality
\[
\int_{T\setminus\gamma F_{}}|H_n f(y)|^2  dy \le C\|f\|_1\lambda
\]
holds.  The constant $C$ is uniform in $f, n$  and  $\lambda$.
\end{lem}

\begin{proof}
Without restriction of generality, in order to avoid writing too many conjugate signs, we suppose that $f$ is a real function.
Let $K\in \F{}, y\in T\setminus\gamma F_{}$, that is, $y\notin\gamma F_{}$ and check the set $K\setminus 3I_{|n_j|}(y)$. Is it possible that
 $K\setminus 3I_{|n_j|}(y)$ is not $K$ or the empty set? Only, when one of $I_{|n_j|}(y), I^+_{|n_j|}(y), I^-_{|n_j|}(y)$ is a subset of $K$ (two dyadic intervals are disjoint or one of them is contained in the other). If so, then it follows that $y\in I_{|n_j|}(y)\subset \gamma K$ which is a contradiction.
 Thus, either $K\setminus 3I_{|n_j|}(y)=\emptyset$ or $K\setminus 3I_{|n_j|}(y)=K$. In both cases
 in the same way as in the proof of Lemma \ref{gammaFjH2} (see (\ref{intK-3I})) (moreover in the first case there is nothing to prove):
\[
\begin{split}
   &
 \left|\int_{K\setminus 3I_{|n|}(y)} f^0(z)\cot\left(\frac{y-z}{2}\right) dz\right|
   \le   \frac{C|K|}{\sin^2\left(\frac{y-z_0}{2}\right)}\int_{K}|f^0(z)| dz
   \le   \frac{C|K|^2\lambda}{\sin^2\left(\frac{y-z_0}{2}\right)},
\end{split}
\]
where $z_0$ is the center of $K$. Similarly,
\[
 \left|\int_{J\setminus 3I_{|n|}(y)} f^0(x)\cot\left(\frac{y-x}{2}\right) dx\right| \le  \frac{C|J|^2\lambda}{\sin^2\left(\frac{y-x_0}{2}\right)},
\]
where $x_0$ is the center of $J$.
This gives
\[
\begin{split}
   &
\int_{T\setminus\gamma F_{}}|H_n f(y)|^2  dy
\le 2\int_{T\setminus\gamma F_{}}|H_n f_0(y)|^2  dy + 2\int_{T\setminus\gamma F_{}}|H_n f^0(y)|^2  dy \le C\|f\|_1\lambda
\\
& + C\sum_{J, K\in \F{}}\int_{T\setminus\gamma F_{}}
\left|\int_{J\setminus 3I_{|n|}(y)} f^0(x)\cot\left(\frac{y-x}{2}\right) dx\right|
  \left|\int_{K\setminus 3I_{|n|}(y)} f^0(z)\cot\left(\frac{y-z}{2}\right) dz\right| dy \\
& \le C\|f\|_1\lambda + C\lambda^2 \sum_{J, K\in \F{}}\int_{T\setminus\gamma F_{}}
\frac{|J|^2 |K|^2}{\sin^2\left(\frac{y-x_0}{2}\right) \sin^2\left(\frac{y-z_0}{2}\right)} dy  =:  C\|f\|_1\lambda + A.
\end{split}
\]
The term $A$ is smaller (or equival to) than the right-hand side in (\ref{A4estimation}) since the integrals are smaller because
$T\setminus \gamma F \subset T\setminus (\gamma J\cup \gamma K)$ for any $J,K\in \F{}$.
Consequently, estimation (\ref{A4estimation2}) can be applied and then $A  \le C\|f\|_1\lambda$.
This completes the proof of Lemma \ref{T-gammaFH}.
\end{proof}

The second lemma in this section is to be proved:
\begin{lem}\label{T-gammaFS}
Let $\gamma >5$ be an odd integer, $l\in\N$ and $f\in L^1(T), \lambda >\|f\|_1/(2\pi)$. Then the inequality
\[
\int_{T\setminus\gamma F}|S_l f(y)|^2  dy \le C\|f\|_1\lambda
\]
holds. The constant $C$   is uniform in $f, l$  and  $\lambda$.
\end{lem}
\begin{proof}
Recall the formula for the Dirichlet kernel (\ref{dirichlet_kernel})
and the definition of the modified partial sums (\ref{modified_partial}) (see Lemma \ref{gammaFjS}).
Also recall the estimation (\ref{^2estimation_on_modified_partial}). That is,
\[
|\tilde S_l f(y) |^2 \le \|f\|_1^2 + \left|H_l(f(\cdot)e^{-\imath(l+1)\cdot})(y)\right|^2 + \left|H_l(f(\cdot)e^{\imath l\cdot})(y)\right|^2.
\]
That is, by Lemma \ref{T-gammaFH} we have
\[
\int_{T\setminus\gamma F}|\tilde S_l f(y)|^2  dy \le C\|f\|_1\lambda.
\]
Recall   the definition of the operator $E_l$ in Lemma \ref{gammaFjS} at (\ref{Eldef}):  
\[
E_lf(y) := l\int_{3I_{|l|}(y)}f(x) dx.
\]
Now we have to check the difference of $S_lf$ and $\tilde S_lf$. It is bounded by (see (\ref{SltildeSldifference}))
\[
\frac{1}{\pi}\left|\int_{3I_{|l|}(y)}f(x)D_l(y-x) dx\right|  \le E_l|f|(y).
\]
That is, we finally have to prove that
\[
\int_{T\setminus\gamma F}|E_l f(y)|^2  dy \le C\|f\|_1\lambda
\]
and apply this inequality for the function $|f|$ (set $\F{}$ for $f$ and $|f|$ is the same). That would complete the proof of Lemma \ref{T-gammaFS}.
Apply Lemma \ref{cz} for the function $f$. We have $f=\sum_{i=0}^{\infty}f_i=f_0+f^0$.
We prove for any $y\in T\setminus\gamma F$ that $E_l f^0(y) =0$.
Let $J\in \F{}$. Then $y\in T\setminus\gamma F$ gives that either $3I_{|l|}(y)\cap J = J$ or $3I_{|l|}(y)\cap J =\emptyset$ because
$I_{|l|}(y), I^+_{|l|}(y), I^-_{|l|}(y)$ can not be a subset of $J$. In both cases
$\int_{3I_{|l|}(y)\cap J}f^0(x) dx = 0$. This holds for each $J\in \F{}$ and consequently,
$l\int_{3I_{|l|}(y)\cap F}f^0(x) dx = E_l f^0(y) = 0$. Finally, by the fact that the operator $E_l$ is of type $(L^2, L^2)$ we have
\[
\int_{T\setminus\gamma F}|E_l f_0(y)|^2  dy \le C\|f\|_1\lambda.
\]
This completes the proof of Lemma \ref{T-gammaFS}.
\end{proof}

 A straightforward consequence of Lemma \ref{T-gammaFS} is

\begin{cor}\label{T-gammaFS-V}
Let $\gamma >5$ be an odd integer,  $n\in\N$. Let $f\in L^1(T), \lambda >\|f\|_1/(2\pi)$.
Then the inequality
\[
\int_{T\setminus\gamma F}|S_{n}f(y)-V_{n}f(y)|^2 dy \le C \|f\|_1\lambda
\]
holds. The constant $C$   is uniform in $f, n$  and  $\lambda$.
\end{cor}
\begin{proof}
  The proof is a direct application of Lemma \ref{T-gammaFS} and follows the steps of the proof of Corollary \ref{gammaFjS-V}.
\end{proof}

Corollaries \ref{gammaFS-V} and \ref{T-gammaFS-V} give
\begin{cor}\label{S-V}
Let $\beta >  7$ be  an odd integer,  $N\in\N$. Let $f\in L^1(T), \lambda >\|f\|_1/(2\pi)$, $(n_j)$ be a lacunary (with parameter not less than $2$) sequence of natural numbers and
$n_j\beta_j = 20(j+1)\log^2 (j+1)$, 
$m_j = \lfloor n_j/10 \rfloor$
for $j=1,\dots, N$.
Then the inequality
\[
\sum_{j=1}^{N}\int_{T}|S_{n_j}f(y)-V_{n_j}f(y)|^2  \left|\sigma_{m_j}1_{\overline{\beta F_{\beta_j}}}(y)\right|^2 dy \le C_{\beta} N\log^5 (N+1)\|f\|_1\lambda
\]
holds. The constant $C_{\beta}$ depends only on $\beta$ and it is uniform in $f,\lambda, N$ and $(n_j)$.
\end{cor}

\section{Orthogonality, replacement and the proof of the main theorem}

The ``orthogonality'' lemma:
\begin{lem}\label{orthogonality}
Let $\beta >  7$ be  an odd integer,  $N\in\N$ and $f\in L^1(T), \lambda >\|f\|_1/(2\pi)$. Let  $(n_j)$ be a lacunary sequence of natural numbers with
$n_{j+1} \ge qn_j$, where $q>2.5$ and $\beta_j$ be the number defined as $n_j\beta_j = 20(j+1)\log^2 (j+1)$ and let $m_j = \lfloor n_j/10\rfloor$ for $j=1,\dots, N$.
Then the relation
\[
\begin{split}
&
\left\|\sum_{j=1}^{N}\left( S_{n_j}f-V_{n_j}f\right)  \left(\sigma_{m_j}1_{\overline{\beta F_{\beta_j}}}\right)\right\|_2^2
\\
& = \sum_{j=1}^{N}\left\||S_{n_j}f-V_{n_j}f|  \left|\sigma_{m_j}1_{\overline{\beta F_{\beta_j}}}\right|\right\|_2^2
 \le C_{\beta} N\log^5 (N+1)\|f\|_1\lambda
\end{split}
\]
holds. The constant $C_{\beta}$ depends only on $\beta$ and it is uniform in $f,\lambda, N$ and $(n_j)$.
\end{lem}
\begin{proof}
  The proof is quite simple and based on the fact that the trigonometric polynomials $(S_{n_j}f-V_{n_j}f)\left(\sigma_{m_j}1_{\overline{\beta F_{\beta_j}}}\right)$
  for different $j$'s are orthogonal because for a $j$ the set of $k$'s for which the $k$th Fourier coefficient is different from zero is a subset
  of
  \[
  \begin{split}
     &
  \set{k\in \mathbb{Z} : k\in [n_j-m_j, 2n_j+m_j]\cup [-2n_j-m_j, -n_j+m_j]}\\
  & \subset \set{k\in \mathbb{Z} : k\in [0.9n_j, 2.1n_j]\cup [-2.1n_j, -0.9n_j]}.
  \end{split}
  \]
   And for example
  $[0.9n_j, 2.1n_j]\cap [0.9n_{j+1}, 2.1n_{j+1}]=\emptyset$ because $2.1n_j< 0.9n_{j+1}$ for every $j$. This proves the equality part of Lemma \ref{orthogonality}. The inequality part of Lemma \ref{orthogonality} is Corollary \ref{S-V}.
\end{proof}

Next, we state and prove the ``replacement'' lemma. That is, - roughly speaking- we show why it is correct to investigate $(S_{n_j}f-V_{n_j}f) \sigma_{m_j}1_{\overline{\beta F_{\beta_j}}}$ instead
of $ S_{n_j}f-V_{n_j}f $. Before this some more notation is needed.

For a number $N\in\mathbb{N}$ let $K = \lfloor\sqrt{N}\rfloor$. That is, $K^2\le N< (K+1)^2$. Besides, for every $0<\delta<1/2$ let $K_0 := \lfloor K^{2\delta}\rfloor$.
For every $1\le i\le N$ there is a unique pair of natural numbers $(j,b)$ such that $i = (j-1)K_0+b$, where
$1\le j  \le N/ K_0 +1$ and $0\le b <  K_0$. Besides, set
\begin{equation}\label{njbprime}
\begin{split}
&n_{j,b}^{\prime} := n_{(j-1)K_0+b} = n_i, \quad
m_{j,b}^{\prime} := m_{(j-1)K_0+b} = m_i = \lfloor n_i/10\rfloor, \\
& \beta_{j,b}^{\prime} := \frac{20(j+1)\log^2(j+1)}{n^{\prime}_{j,b}} = \beta_i^{\prime}.
\end{split}
\end{equation}

Moreover, set
\[
T_{N,\beta} f:= \frac{1}{N}\sum_{i=1}^{N}\left( S_{n_i}f-V_{n_i}f\right)  \left(\sigma_{m_i}1_{\overline{\beta F_{\beta_i^{\prime}}}}\right).
\]

The ``replacement'' lemma:

\begin{lem}\label{replacement}
Let $\beta > 7$ be an  odd integer, $0<\delta<1/2$, $f\in L^1(T), \lambda >\|f\|_1/(2\pi)$. Let  $(n_j)$ be a strictly monotone increasing sequence of natural numbers,
\[
T_{N}f := \frac{1}{N}\sum_{i=1}^{N}
\left( S_{n_i}f-V_{n_i}f\right).
\]
Then we have
\[
\m{y\in T : \sup_{N\in \mathbb{N}}|T_{N, \beta}f(y) - T_{N}f(y)|>\lambda/2} \le C_{\beta}\frac{1}{1-2\delta}\sqrt{\frac{\|f\|_1}{\lambda}}.
\]

That is, the maximal operator of $|T_{N, \beta}f - T_{N}f|$ is a kind of weak type $(L^1, L^1)$.
The constant $C_{\beta}$ depends only on $\beta$ and it is uniform in $f,\lambda$ and $(n_j)$.
\end{lem}

\begin{proof}

Investigate the $i$th addend in $T_{N}f(y)-T_{N, \beta}f(y)$, that is,
$\left(S_{n_i}f-V_{n_i}f\right)\left(1-\sigma_{m_i}1_{\overline{\beta F_{\beta_i^{\prime}}}}\right)$. We give an estimation for
$1-\sigma_{m_i}1_{\overline{\beta F_{\beta_i^{\prime}}}}$. Since $\sigma_{m_i}1 = 1$ everywhere, one has
\[
\begin{split}
     &
1-\sigma_{m_i}1_{\overline{\beta F_{\beta^{\prime}_i}}}(y)  = \sigma_{m_i}1_{\beta F_{\beta^{\prime}_i}}(y)
= \frac{1}{\pi}\int_{-\pi}^{\pi}  1_{\beta F_{\beta^{\prime}_i}}(x) K_{m_i}(y-x) dx \\
     & = \frac{1}{\pi}\int_{\beta F_{\beta^{\prime}_i}} K_{m_i}(y-x) dx
     \le \frac{C}{m_i}\int_{\beta F_{\beta^{\prime}_i}}\frac{1}{|y-x \, (\Mod T)|^2}dx,  \\
     \end{split}
\]
where $y-x\, (\Mod T)\in T$ and  $y-x\, (\Mod T)= y-x + u2\pi$ for a $u \in\{-1,0,1\}$. That is, if $y-x$ is not in interval $T$, then it is shifted by $2\pi$.
This can be done, since $\sin^2 ((y-x)/2) = \sin^2 ((y-x + u2\pi)/2)$. Besides, notice that $0\le |(y-x\, (\Mod T))/2| \le  \pi/2$.
Let $\alpha>\beta>7$ be odd integers and let $y\in T\setminus \alpha  F_{}$.
Thus, for each $i$ we have  $y\in T\setminus \alpha F_{\beta^{\prime}_i}$.
If $x\in \beta F_{\beta^{\prime}_i}$, then there is an $I\in \F{\beta^{\prime}_i}$  (that is, $I\in \F{}, |I|>\beta^{\prime}_i$) such that $x\in \beta I$.
This gives $|y-x \, (\Mod T)|>(\alpha-\beta)|I|/2 > (\alpha-\beta)\beta^{\prime}_i/2$.
Thus,

\[
\begin{split}
     &
0\le 1-\sigma_{m_i}1_{\overline{\beta F_{\beta^{\prime}_i}}}(y) \\
     & \le \frac{C}{m_i}\int_{\set{z: z> (\alpha-\beta)\beta^{\prime}_i/2}}\frac{1}{z^2} dz \\
     & \le \frac{C}{\alpha-\beta}\frac{1}{m_i \beta^{\prime}_i}
     = \frac{C}{\alpha-\beta}\frac{1}{m^{\prime}_{j,b}\beta^{\prime}_{j,b}}\\
     & \le \frac{C}{\alpha-\beta}\frac{1}{(j+1)\log^2 (j+1)}.
\end{split}
\]
This gives for  $y\in T\setminus \alpha  F$
\[
\begin{split}
     & A:= \m{y\in T\setminus\alpha F: \sup_{N\in \mathbb{N}}|T_{N, \beta}f(y) - T_{N}f(y)|>\lambda/2}\\
     & \le \mes\Biggl\{y\in T\setminus\alpha F: \sup_N \frac{1}{K^2}\sum_{b=0}^{K_0}\sum_{j=1}^{\lfloor N/K_0\rfloor +1}
     |S_{n^{\prime}_{j,b}}f(y) - V_{n^{\prime}_{j,b}}f(y)|\frac{C}{\alpha-\beta}\\
     & \times \frac{1}{(j+1)\log^2 (j+1)} >\lambda/2\Biggr\}\\
& \le \frac{2}{\lambda}\int_{T\setminus \alpha F}
\sup_N \frac{1}{K^2}\sum_{b=0}^{K_0}\sum_{j=1}^{\infty}
     |S_{n^{\prime}_{j,b}}f(y) - V_{n^{\prime}_{j,b}}f(y)|\frac{C}{\alpha-\beta}\frac{1}{(j+1)\log^2 (j+1)} dy.
\end{split}
\]
Use Corollary \ref{T-gammaFS-V} and the Cauchy-Bunyakovsky-Schwarz inequality. Then for every $l\in \mathbb{N}$
\[
\int_{T\setminus\alpha F}|S_{l}f(y)-V_{l}f(y)| dy \le \sqrt{C\|f\|_1\lambda}
\]
and consequently by $K^2 \le N < (K+1)^2, K_0 \le K^{2\delta}$ we have
\[
\begin{split}
     & A \le
\frac{C}{\lambda(\alpha-\beta)}\sum_{K=1}^{\infty}\frac{1}{K^2}\sum_{b=0}^{\lfloor K^{2\delta}\rfloor}\sum_{j=1}^{\infty}\sqrt{C\|f\|_1\lambda}
\frac{1}{(j+1)\log^2 (j+1)}\\
& \le \frac{C}{(\alpha-\beta)(1-2\delta)}\sqrt{\frac{\|f\|_1}{\lambda}}
\end{split}
\]
since 
$2\delta<1$.
Since $\alpha$ can be any odd integer with $\alpha>\beta$, say $\beta + 2$, then by $\mes (\alpha F)\le \alpha\|f\|_1/\lambda$
the  proof of Lemma \ref{replacement} is complete.
\end{proof}

\medskip\noindent
\textit{Proof of the main theorem (Theorem \ref{main}).}
Basically, we use the notation of Lemma \ref{replacement} and (\ref{njbprime}). 
We prove that the sequence $(n_{(j-1)K_0+b})$ 
is lacunary for any fixed $b < K_0$ as $j$ runs from $1$ in a way that
$(j-1)K_0+b$ 
is less than $(K+1)^2$. This observation follows, from
\[
n_{a + K_0} \ge  \left(1 + \frac{1}{(a+K_0)^{\delta}}\right)^{K_0}n_a
> \left(1 + \frac{1}{(1+K)^{2\delta}}\right)^{K^{2\delta}-1}n_a \ge 2.6 n_a
\]
for every $a\in \mathbb{N}$ with $a + K_0 \le (K+1)^2$
for $K\ge k_{\delta}$ for some fixed $k_{\delta}$ because
\[
\left(1 + \frac{1}{(1+K)^{2\delta}}\right)^{K^{2\delta}-1} \to \exp(1)
\]
as $K\to\infty$. 
Consequently Lemma \ref{orthogonality} can be applied to sequence $(n_{j, b}^{\prime})$. Before this,
apply the well-known inequality between the arithmetic and quadratic means. (Suppose that $N\ge k_{\delta}^2$.)

\[
\begin{split}
& \left|T_{N, \beta}f(y)\right|^2
\le \frac{2}{K^4}
\left|\sum_{i=1}^{K^2}\left( S_{n_i}f-V_{n_i}f\right)  \left(\sigma_{m_i}1_{\overline{\beta F_{\beta^{\prime}_i}}}\right)\right|^2 \\
& +
\frac{2}{K^4}
\left|\sum_{i=K^2+1}^{N}\left( S_{n_i}f-V_{n_i}f\right)  \left(\sigma_{m_i}1_{\overline{\beta F_{\beta^{\prime}_i}}}\right)\right|_2^2
=: A_N + B_N.
\end{split}
\]
Again, by  the inequality between the arithmetic and quadratic means we have
\[
\begin{split}
    & A_N  
     \le \frac{C K_0}{K^4}\sum_{b< K_0}
     \sum_{L=\lfloor K^2/K_0\rfloor}^{\lfloor K^2/K_0\rfloor +1}\left|\sum_{j=1}^{L}
     \left( S_{n^{\prime}_{j, b}}f-V_{n^{\prime}_{j, b}}f\right)  \left(\sigma_{m^{\prime}_{j, b}}1_{\overline{\beta F_{\beta^{\prime}_{j, b}}}}\right)\right|^2 =:A^{1}_K.
\end{split}
\]
Apply Lemma \ref{orthogonality}.
\[
\|A^1_K\|_1 \le C_{\beta}\frac{K_0}{K^4}\sum_{b < K_0}
\frac{K^2}{K_0}\log^5 \left(\frac{K^2}{K_0}+1\right)\|f\|_1\lambda
 \le C_{\beta}\frac{1}{K^{2(1-\delta)}}\log^5 (K+1) \|f\|_1\lambda.
\]

Fix natural numbers $K,N$ such that $K^2\le N < (K+1)^2$, where $K\ge k_{\delta}$.
Apply again the  inequality between the arithmetic and quadratic means.
\[
\begin{split}
    & B_N
 \le \frac{C K_0}{K^4}\sum_{b< K_0}\sum_{L=\lfloor (N-K^2)/K_0\rfloor}^{\lfloor (N-K^2)/K_0\rfloor +1}
     \left|\sum_{j=1}^{L}
     \left( S_{n^{\prime}_{j, b}}f-V_{n^{\prime}_{j, b}}f\right) \left(\sigma_{m^{\prime}_{j, b}}
     1_{\overline{\beta F_{\beta^{\prime}_{j, b}}}}\right)\right|^2.
\end{split}
\]

Apply Lemma \ref{orthogonality} (and the fact that $ N-K^2 < 2K$).
\[
\|B_N\|_1 \le C_{\beta}\frac{K_0}{K^4}\sum_{b < K_0}
\frac{K}{K_0}\log^5 \left(\frac{N}{K_0}+1\right)\|f\|_1\lambda
 \le C_{\beta}\frac{1}{K^{1+2(1-\delta)}}\log^5 (K+1) \|f\|_1\lambda.
\]
Consequently, for  $B_K^1 := \sup_{K^2\le N < (K+1)^2} B_N$  ($k_{\delta}\le K\in \mathbb{N}$ is fixed) 
we have
\[
\begin{split}
\|B^1_K\|_1 \le  \sum_{N=K^2}^{(K+1)^2-1}C_{\beta}\frac{1}{K^{1+2(1-\delta)}}\log^5 (K+1) \|f\|_1\lambda \le
C_{\beta}\frac{1}{K^{2(1-\delta)}}\log^5 (K+1) \|f\|_1\lambda.
\end{split}
\]

This immediately gives   ($\delta<1/2$ is an arbitrarily  fixed number) that
\[
\begin{split}
    & \m{y\in T : \sup_{k^2_{\delta}\le N\in \mathbb{N}}\left|T_{N, \beta}f(y)\right|>\lambda/2} \\
    & \m{y\in T : \sup_{k^2_{\delta}\le N\in \mathbb{N}}\left|T_{N, \beta}f(y)\right|^2 >\lambda^2/4} \\
    & \m{y\in T : \sup_{k^2_{\delta}\le K\in \mathbb{N}} \left(A_K^1 + B_K^1\right) >\lambda^2/4} \\
&
 \le \sum_{K=1}^{\infty}C_{\beta}\frac{1}{K^{2(1-\delta)}}\log^5 (K+1) \frac{\|f\|_1}{\lambda}
 \le C_{\beta,\delta}\frac{\|f\|_1}{\lambda}
\end{split}
\]
because $2(1-\delta)>1$.

This inequality for $\sup_N |T_{N,\beta}|$  and  (``replacement'') Lemma \ref{replacement}  by the fact that $\beta$ can be any odd integer greater than $7$, say $9$, give that
\[
\m{y\in T : \limsup_{N\in \mathbb{N}}|T_{N}f(y)|>\lambda}
\le \m{y\in T : \sup_{k^2_{\delta}\le N\in \mathbb{N}}|T_{N}f(y)|>\lambda} \le C_{\delta}\sqrt{\frac{\|f\|_1}{\lambda}}.
\]
Let $\epsilon$ and $\eta$ be positive reals discussed later.
The set of trigonometric polynomials is dense in $L^1$. Thus, we have a trigonometric polynomial $P$ such that $\|f-P\|_1\le \eta, 2\pi\epsilon$.
Besides, for $P$ we also have that $S_{n_i}P-V_{n_i}P  = 0$ holds for sufficiently large $i$.
Therefore,
\[
\begin{split}
& \m{y\in T : \limsup_{N\in \mathbb{N}}|T_{N}f(y)|>\epsilon} \\
& \le \m{y\in T : \limsup_{N\in \mathbb{N}}|T_{N}(f-P)(y)|  + \limsup_{N\in \mathbb{N}}|T_{N}P(y)| >\epsilon}\\
& = \m{y\in T : \limsup_{N\in \mathbb{N}}|T_{N}(f-P)(y)| >\epsilon} \\
&  \le C_{\delta}\sqrt{\frac{\|f-P\|_1}{\epsilon}} \\
& \le C_{\delta}\sqrt{\frac{\eta}{\epsilon}}
\end{split}
\]
for each $\eta>0$ and consequently
\[
\m{y\in T : \limsup_{N\in \mathbb{N}}|T_{N}f(y)|>\epsilon} = 0
\]
for each $\epsilon>0$.
Thus, by
\[
\set{y\in T : \limsup_{N\in \mathbb{N}}|T_{N}f(y)|>0} \subset \bigcup_{l=1}^{\infty}
\set{y\in T : \limsup_{N\in \mathbb{N}}|T_{N}f(y)|> 1/l}
\]
we proved for each integrable function $f$ the a.e. relation
\[
\lim_{N\to \infty}\frac{1}{N}\sum_{j=1}^{N} \left(S_{n_j}f-V_{n_j}f\right) = 0.
\]
Since for the de la Vall\'ee-Poussin means $V_{n_j}f$  the a.e. relation
\[
\lim_{j\to \infty}V_{n_j}f = f
\]
is well-known, the proof of the main theorem is complete.

\qed

\textbf{Acknowledgement.} The author is deeply indebted to the anonymous referees for finding some errors in the first version of the manuscript and for their valuable help.

\providecommand{\bysame}{\leavevmode\hbox to3em{\hrulefill}\thinspace}
\providecommand{\MR}{\relax\ifhmode\unskip\space\fi MR }
\providecommand{\MRhref}[2]{%
  \href{http://www.ams.org/mathscinet-getitem?mr=#1}{#2}
}
\providecommand{\href}[2]{#2}

\end{document}